\documentclass[journal]{IEEEtran}
\usepackage{mathtools}
\usepackage[normalem]{ulem} 

\usepackage{amssymb}
\usepackage{mathrsfs}

\newcommand{\cut}[1]{{}}

\newcommand{\va}{{\mathbf{a}}}
\newcommand{\vb}{{\mathbf{b}}}

\newcommand{\vd}{{\mathbf{d}}}

\newcommand{\vp}{{\mathbf{p}}}
\newcommand{\vq}{{\mathbf{q}}}

\newcommand{\vw}{{\mathbf{w}}}
\newcommand{\vx}{{\mathbf{x}}}
\newcommand{\vy}{{\mathbf{y}}}
\newcommand{\vz}{{\mathbf{z}}}

\newcommand{\vA}{{\mathbf{A}}}
\newcommand{\vB}{{\mathbf{B}}}

\newcommand{\vI}{{\mathbf{I}}}

\newcommand{\vL}{{\mathbf{L}}}
\newcommand{\vM}{{\mathbf{M}}}

\newcommand{\vQ}{{\mathbf{Q}}}

\newcommand{\vS}{{\mathbf{S}}}

\newcommand{\vU}{{\mathbf{U}}}

\newcommand{\vW}{{\mathbf{W}}}

\newcommand{\cE}{{\mathcal{E}}}

\newcommand{\cG}{{\mathcal{G}}}

\newcommand{\cN}{{\mathcal{N}}}

\newcommand{\cU}{{\mathcal{U}}}
\newcommand{\cV}{{\mathcal{V}}}

\newcommand{\RR}{\mathbb{R}}

\newcommand{\vzero}{\mathbf{0}}
\newcommand{\vone}{{\mathbf{1}}}

\newcommand{\st}{{\text{s.t.}}} 
\newcommand{\Diag}{{\mathrm{Diag}}} 
\newcommand{\range}{{\mathbf{range}}} 
\newcommand{\tr}{{\mathrm{tr}}} 

\newcommand{\prox}{\mathbf{prox}}

\newcommand{\Null}{\mathbf{Null}}
\newcommand{\Span}{\mathbf{span}}

\makeatletter
\let\@@span\span
\def\sp@n{\@@span\omit\advance\@multicnt\m@ne}
\makeatother

\DeclareMathOperator*{\argmin}{arg\,min}

\DeclareMathOperator*{\Min}{minimize}

\newcommand{\bc}{\begin{center}}
\newcommand{\ec}{\end{center}}

\newcommand{\bdm}{\begin{displaymath}}
\newcommand{\edm}{\end{displaymath}}

\newcommand{\beq}{\begin{equation}}
\newcommand{\eeq}{\end{equation}}

\newcommand{\bfl}{\begin{flushleft}}
\newcommand{\efl}{\end{flushleft}}

\newcommand{\bt}{\begin{tabbing}}
\newcommand{\et}{\end{tabbing}}

\newcommand{\beqn}{\begin{align}}
\newcommand{\eeqn}{\end{align}}

\newcommand{\beqs}{\begin{align*}} 
\newcommand{\eeqs}{\end{align*}}  

\newtheorem{theorem}{Theorem}
\newtheorem{assumption}{Assumption}
\newtheorem{lemma}{Lemma}
\newtheorem{proposition}{Proposition}

\usepackage[top=1.3in, bottom=1.3in, left=1.0in, right=1.0in]{geometry}

\usepackage{graphicx}
\usepackage{amsmath, amssymb, bm}
\usepackage{cases}
\usepackage{float, subfigure}
\usepackage{color}
\usepackage{amsfonts}
\usepackage{cite, url}
\usepackage{color,hyperref}
\definecolor{darkblue}{rgb}{0.0,0.0,0.6}
\definecolor{mypink1}{rgb}{0.858, 0.188, 0.478}
\hypersetup{colorlinks,breaklinks,
            linkcolor=darkblue,urlcolor=darkblue,
            anchorcolor=darkblue,citecolor=darkblue}
						
\usepackage{algorithmic}
\usepackage{algorithm}
\usepackage{alphalph}
\usepackage{scalerel}

\newtheorem{remark}{Remark}

\begin{document}

\newcommand{\TheAuthors}{Zhi Li, Wei Shi, and Ming Yan}
\newcommand{\Lap}{\textbf{{\L}}}

\title{A Decentralized Proximal-Gradient Method with Network Independent Step-sizes and Separated Convergence Rates}


\author{Zhi~Li,
        Wei~Shi,
        and~Ming~Yan
\thanks{Zhi Li is with Department of Computational Mathematics, Science and Engineering, Michigan State University, East Lansing, MI, 48824, USA. (zhili@msu.edu)}
\thanks{Wei Shi was with 
Department of Electrical Engineering, Princeton University, Princeton, NJ 08544, USA.}

\thanks{Ming Yan is with Department of Computational Mathematics, Science and Engineering and Department of Mathematics, Michigan State University, East Lansing, MI, 48824, USA. (myan@msu.edu)}
\thanks{This work is supported in part by the NSF grant DMS-1621798.}}

%
%

\markboth{Journal of TSP \LaTeX\ Class Files,~Vol.~x, No.~x, Sep~2018}%
{Shell \MakeLowercase{\textit{et al.}}: Bare Demo of IEEEtran.cls for IEEE Journals}
%



\maketitle

\begin{abstract}
This paper proposes a novel proximal-gradient algorithm for a decentralized optimization problem with a composite objective containing smooth and non-smooth terms.
Specifically, the smooth and nonsmooth terms are dealt with by gradient and proximal updates, respectively. 
The proposed algorithm is closely related to a previous algorithm, PG-EXTRA \cite{shi2015proximal}, but has a few advantages. 
First of all, agents use uncoordinated step-sizes, and the stable upper bounds on step-sizes are independent of network topologies. 
The step-sizes depend on local objective functions, and they can be as large as those of the gradient descent. 
Secondly, for the special case without non-smooth terms, linear convergence can be achieved under the strong convexity assumption.  
The dependence of the convergence rate on the objective functions and the network are separated, and the convergence rate of the new algorithm is as good as one of the two convergence rates that match the typical rates for the general gradient descent and the consensus averaging. 
We provide numerical experiments to demonstrate the efficacy of the introduced algorithm and validate our theoretical discoveries.
\end{abstract}


%
\IEEEpeerreviewmaketitle

\section{Introduction}\label{sec:intro}

\IEEEPARstart{T}{his} paper focuses on the following decentralized optimization problem:
\begin{equation} \label{eq:F}
\Min\limits_{x\in\RR^p}~\bar f(x):=\frac{1}{n}\sum\limits_{i=1}^n (s_i(x)+r_i(x)),
\end{equation}
where $s_i:\RR^p\rightarrow \RR$ and $r_i: \RR^p\rightarrow \RR\cup\{+\infty\}$ are two lower semi-continuous proper convex functions held privately by agent $i$ to encode the agent's objective. 
We assume that one function (e.g., without loss of generality, $s_i$) is differentiable and has a Lipschitz continuous gradient with parameter $L>0$, and the other function $r_i$ is proximable, i.e., its proximal mapping 
$$   \prox_{\lambda r_i}(y) = \argmin_{x\in \RR^p}~\lambda r_i(x) + {1\over2}\|x-y\|^2,$$
has a closed-form solution or can be computed easily. 
Examples of $s_i$ include linear functions, quadratic functions, and logistic functions, while $r_i$ could be the $\ell_1$ norm, 1D total variation, or indicator functions of simple convex sets. 
In addition, we assume that the agents are connected through a fixed bi-directional communication network. 
Every agent in the network wants to obtain an optimal solution of~\eqref{eq:F} while it can only receive/send nonsensitive messages\footnote{We believe that agent $i$'s instantaneous estimation on the optimal solution is not a piece of sensitive information but $s_i$ and $r_i$ are.} from/to its immediate neighbors. 

Specific problems of form~\eqref{eq:F} that require a decentralized computing architecture have appeared in various areas including networked multi-vehicle coordination, distributed information processing and decision making in sensor networks, as well as distributed estimation and learning. 
Some examples include distributed average consensus~\cite{Xiao2007,Cai2014,Olshevsky2014}, distributed spectrum sensing~\cite{Bazerque2010}, information control~\cite{Ren2006,Olshevsky2010}, power systems control~\cite{Ram2009,Gan2013}, statistical inference and learning~\cite{Rabbat2004,Forero2010,nedic2017fast}. 
In general, decentralized optimization fits the scenarios where the data is collected and/or stored in a distributed network, a fusion center is either inapplicable or unaffordable, and/or computing is required to be performed in a distributed but collaborative manner by multiple agents or by network designers.

\subsection{Literature Review}
The study on distributed algorithms dates back to the early 1980s~\cite{Bertsekas1983,Tsitsiklis1986}. 
Since then, due to the emergence of large-scale networks, decentralized (optimization) algorithms, as a special type of distributed algorithms for solving problem~\eqref{eq:F}, have received  attention. 
Many efforts have been made on star networks with one master agent and multiple slave agents~\cite{Boyd2011,Cevher2014}. 
This scheme is ``centralized'' due to the use of a ``master'' agent. 
It may suffer a single point of failure and may violate the privacy requirement in certain applications. 
In this paper, we focus on solving~\eqref{eq:F} in a decentralized fashion, where no ``master'' agent is used. 

Incremental algorithms~\cite{Nedic2000,Nedic2001a,Nedic2001b,Ram2009_2,Bertsekas2011,WangB2013} can solve~\eqref{eq:F} without the need of a ``master'' agent and it is based on a directed ring network.  
To handle general (possibly time-varying) networks, the distributed sub-gradient algorithm was proposed in~\cite{Nedic2009}. 
This algorithm and its variants~\cite{Ram2010,Nedic2011} are intuitive and simple but usually slow due to the diminishing step-size that is needed to obtain a consensual and optimal solution, even if the objective functions are differentiable and strongly convex. 
With a fixed step-size, these distributed methods can be fast, but they only converge to a neighborhood of the solution set which depends on the step-size. 
This phenomenon creates an exactness-speed dilemma~\cite{Kun2014}. 

A class of distributed approaches that bypass this dilemma is based on introducing the Lagrangian dual. 
The resulting algorithms include distributed dual decomposition~\cite{Terelius2011} and decentralized alternating direction method of multipliers (ADMM)~\cite{Bertsekas1997}. 
The decentralized ADMM and its proximal-gradient variant can employ a fixed step-size to achieve $O(1/k)$ rate under general convexity assumptions~\cite{Wei2013,Chang2014,Hong2015}. 
Under the strong convexity assumption, the decentralized ADMM has linear convergence for time-invariant undirected graphs~\cite{Shi2014}. 
There exist some other distributed methods that do not (explicitly) use dual variables but can still converge to an optimal consensual solution with fixed step-sizes. 
In particular, works in~\cite{Chen2012_2,Jakovetic2014} employ multi-consensus inner loops, Nesterov's acceleration, and/or the adapt-then-combine (ATC) strategy. 
Under the assumption that the objectives have bounded and Lipschitz continuous gradients\footnote{This means that the nonsmooth terms $r_i$'s are absent. Such assumption is much stronger than the one used for achieving the $O(1/k^2)$ rate in Nesterov's optimal gradient method~\cite{nesterov2013introductory}.}, the algorithm proposed in~\cite{Jakovetic2014} has $O\left(\ln(k)/k^2\right)$ rate. 
References~\cite{Shi2015,shi2015proximal} use a difference structure to cancel the steady state error in decentralized gradient descent~\cite{Nedic2009,Kun2014}, thereby developing the algorithm EXTRA and its proximal-gradient variant PG-EXTRA. 
It converges at an $O(1/k)$ rate when the objective function in~\eqref{eq:F} is convex and has a linear convergence rate when the objective function is strongly convex and $r_i(x)=0$ for all $i$.

A number of recent works employed the so-called gradient tracking~\cite{Zhu2010} to conquer different issues~\cite{Xu2015,Lorenzo2016,Qu2016,nedich2016achieving,nedic2016geometrically}. 
To be specific, works~\cite{Xu2015,nedic2016geometrically} relax the step-size rule to allow uncoordinated step-sizes across agents. 
Paper~\cite{Lorenzo2016} solves non-convex optimization problems. 
Paper~\cite{nedich2016achieving} aims at achieving geometric convergence over time-varying graphs. 
Work~\cite{Qu2016} improves the convergence rate over EXTRA, and its formulation is the same as that in~\cite{nedic2016geometrically}. 

Another topic of interest is decentralized optimization over directed graphs~\cite{Nedic2013,Xi2015,Zeng2015,sun2016distributed,nedich2016achieving}, which is beyond the scope of this paper. 

\subsection{Proposed Algorithm and its Advantages}
To proceed, let us introduce some basic notation first. 
Agent $i$ holds a local variable $x_i \in \RR^{p}$, and we denote $x_i^k$ as its value at the $k$-th iteration. 
Then, we introduce a new function that is the average of all the local functions with local variables as
\begin{equation}
{f}(\vx):=\frac{1}{n}\sum_{i=1}^{n}(s_{i}(x_i)+r_i(x_i)),
\end{equation}
where 
\begin{equation}\label{eq:define_x}
\vx:=
\begin{pmatrix}
	-& x_1^\top &- \\ 
-	& x_2^\top &- \\ 
	& \vdots & \\ 
	-& x_n^\top &- 
\end{pmatrix}
\in \RR^{n\times p}.
\end{equation}
If all local variables are identical, i.e., $x_{1}=\dots=x_{n}$, we say that $\vx$ is consensual.
In addition, we define
\begin{align}\label{eq:define_sr}
{s}(\vx):=\frac{1}{n}\sum_{i=1}^{n}s_{i}(x_i), \quad
{r}(\vx):=\frac{1}{n}\sum_{i=1}^{n}r_{i}(x_i).
\end{align}
We have $f(\vx)=s(\vx)+r(\vx)$. 
The gradient of $s$ at $\vx$ is given in the same way as $\vx$ in~\eqref{eq:define_x} by 
\begin{equation}
\nabla s (\vx):=
\begin{pmatrix}
-& \left(\nabla s_{1}(x_{1})\right)^{\top} &- \\ 
-	& \left(\nabla s_{2}(x_{2})\right)^{\top} &- \\ 
& \vdots & \\ 
-& \left(\nabla s_{n}(x_{n})\right)^{\top} &- 
\end{pmatrix}
\in \RR^{n\times p}.
\end{equation}

\footnotetext{In the original EXTRA, two mixing matrices $\vW$ and $\widetilde{\vW}$ are used. For simplicity, we let $\widetilde{\vW}=\frac{I+\vW}{2}$ here.}

By making a simple modification over PG-EXTRA~\cite{shi2015proximal}, our proposed algorithm brings a big improvement in the speed and the dependency of convergence over networks. 
To better expose this simple modification, let us compare a special case of our proposed algorithm with EXTRA for the smooth case, i.e., $r(\vx)=0$. 
\begin{subequations}\label{eq:compare1}
\begin{align}
\begin{split}
(\textnormal{EXTRA}\footnotemark)\ \vx^{k+2} = {\vI+\vW\over 2} (2\vx^{k+1} - \vx^k)  \\
- \alpha \nabla s(\vx^{k+1}) +\alpha \nabla s(\vx^k),\\
\end{split}\\
\begin{split}
(\textnormal{Proposed NIDS})\ \vx^{k+2} =  {\vI+\vW\over 2} (2\vx^{k+1} - \vx^k  \\
- \alpha \nabla s(\vx^{k+1}) +\alpha \nabla s(\vx^k)). 
\end{split}\label{eq:compare1b}
\end{align}
\end{subequations}
Here, $\vW\in\RR^{n\times n}$ is a  matrix that represents information exchange between neighboring agents (more details about this matrix are in Assumption~\ref{assu:1}) and $\alpha$ is the step-size. 
The only difference between EXTRA and the proposed algorithm is the information exchanged between the agents. 
EXTRA exchanges only the estimations $2\vx^{k+1}-\vx^k$, while the proposed algorithm exchanges the gradient adapted estimations, i.e., $2\vx^{k+1} - \vx^k  - \alpha \nabla s(\vx^{k+1})+\alpha \nabla s(\vx^k)$.
Because of this small modification, the proposed algorithm has a \underline{N}etwork \underline{I}n\underline{D}ependent \underline{S}tep-size, which will be explained later. Therefore we name the proposed algorithm NIDS and will use this abbreviation throughout the paper. For the nonsmooth case, more detailed comparison between PG-EXTRA and NIDS will be given in Section~\ref{sec:alg}.




\begin{table*}[!ht]
\centering
\caption{ Summary of algorithmic parameters used in EXTRA and NIDS. The step size bound on $\alpha$ for EXTRA comes from reference \cite{li2017aprimal} which improves that given in reference~\cite{Shi2015}.\label{tab:comp_EXTRA_NIDS} 
}
		\vspace{-0.5em}
		\hspace{-4.7em}\begin{tabular}{c|c|c|c|c} 
			\hline\hline
			Method& $\alpha$ or $\alpha_i$ & $c$ \\
			\hline\hline
			EXTRA ($\lambda_{n}(\vW)$ given)&$\alpha<(5+3 \lambda_{n}(\vW))/(4\max_i L_i)$&--\\
			\hline
			{EXTRA ($\lambda_{n}(\vW)$ not given)}&${\alpha<1/(2\max_i L_i)}$&--\\			
			\hline
			NIDS ($\lambda_{n}(\vW)$ given)&$\alpha_i<2/L_i$&$c\leq 1/ (( 1-\lambda_n(\vW))\max_i{\alpha_i})$\\
			\hline            
			NIDS ($\lambda_{n}(\vW)$ not given)&$\alpha_i<2/L_i$&$ c\leq {1/ (2\max_i{\alpha_i})}$\\
			\hline
		\end{tabular}
\end{table*}

{\bf A large and network independent step-size for NIDS:} 
All works mentioned above either employ pessimistic step-sizes or have network dependent upper bounds on step-sizes. 
Furthermore, the step-sizes for the strongly convex case are more conservative.
For example, the step-size used to achieve linear convergence rates for EXTRA in~\cite{Shi2015,2019arXiv190401196A} is in the order of $O(\mu/L^2)$, where $\mu$ and $L$ are the strong convexity constant of $s(\vx)$ and Lipschitz constant of $\nabla s(\vx)$, respectively. 
As a contrast, the centralized gradient descent can choose a step-size in the order of $O(1/L)$.
The upper bound of step-size for EXTRA was recently improved to $(5+3\lambda_n(\vW))/(4L)$ in~\cite{li2017aprimal}. 
We can choose $1/(2L)$ for any $\vW$ satisfying Assumption~\ref{assu:1}.
Another example of employing a constant step-size in distributed optimization is DIGing~\cite{nedich2016achieving}. 
Although its ATC variant in~\cite{nedic2016geometrically} was shown to converge faster than DIGing, the step-size is still very conservative compared to $O(1/L)$. 
We will show that the step-size of NIDS can have the same upper bound $2/L$ as that of the centralized gradient descent. 
The achievable step-sizes of NIDS for $o(1/k)$ rate in the general convex case and the linear convergence rate in the strongly convex case are both in the order of $O(1/L)$. 
Furthermore, NIDS allows each agent to have an individual step-size. 
Each agent $i$ can choose a step-size $\alpha_i$ that is as large as $2/L_i$ on any connected network, where $L_i$ is the Lipschitz constant of $\nabla s_i(x)$ and $L_i\leq L$ for all $i$ ($L=\max_iL_i$). 
Apart from the step-sizes, to run NIDS, a common/public parameter $c$ is needed for the construction of $\widetilde{\vW}$ (see~\eqref{for:alg1} for the algorithm and $\widetilde{\vW}$). 
This parameter $c$ can be chosen without any knowledge of the network (or the mixing matrix $\vW$). For example, $c={1}/({2\max_i\alpha_i})$. Table~\ref{tab:comp_EXTRA_NIDS} provides an overview of algorithmic configurations for EXTRA and NIDS. NIDS works as long as each agent can estimate its local functional parameter: No agent needs any other global information including the number of agents in the whole network except the largest step-size, if it is not the same for all agents.

In the line of research of optimization over heterogeneous networks, after the initial work~\cite{Xu2015} regarding uncoordinated step sizes, references~\cite{yuan2017exact1} and~\cite{yuan2017exact2} introduce and analyze a diffusion strategy with corrections that achieves an exact linear convergence with uncoordinated step sizes. This exact diffusion algorithm is still related to the Lagrangian method but can be considered as having incorporated a CTA structure. The CTA strategy can, similar to the ATC strategy, improve the convergence speed of certain consensus optimization algorithms (see~\cite[Remark 3]{nedich2016achieving} and~\cite[Section II.C]{sun2016distributed}). However, the analysis in~\cite{yuan2017exact2}, though allowing step size mismatch across the network, does not take into consideration the heterogeneity of agents' functional conditions. 
Furthermore, their upper bound for the step-size is in the order of $O(\mu/L^2)$.



{\bf Sublinear convergence rate for the general case:}
Under the general convexity assumption, we show that NIDS has a convergence rate of $o(1/k)$, which is slightly better than the $O(1/k)$ rate of PG-EXTRA. 
Because the step-size of NIDS does not depend on the network topology and is much larger than that of PG-EXTRA, NIDS can be much faster than PG-EXTRA, as shown in the numerical experiments. 

{\bf Linear convergence rate for the strongly convex case:} 
Let us first define ``scalability''. 
When the iterate $x^k$ of an algorithm converges to the optimal solution $x^*$ linearly, i.e., $\|x^k-x^*\|^2=O((1-1/S)^k)$ with some positive constant $S$, we say that the algorithm needs to run $O(S) \log(1/\epsilon)$ iterations to reach $\epsilon$-accuracy. 
So we call $O(S)$ the scalability of the algorithm.

For the case where the non-smooth terms are absent and the functions $\{s_i\}_{i=1}^n$ are strongly convex, we show that NIDS achieves a linear convergence rate whose dependencies on the functions $\{s_i\}_{i=1}^n$ and the network topology are decoupled. 
To be specific, to reach $\epsilon$-accuracy, the number of iterations needed for NIDS is
\[O\left(\max\left(\frac{L}{\mu},\frac{1-\lambda_n(\vW)}{1-\lambda_2(\vW)}\right)\right)\log\frac{1}{\epsilon},\] 
where $\lambda_i(\vW)$ is the $i$th largest eigenvalue of $\vW$. 
Both $\frac{L}{\mu}$ and $\frac{1-\lambda_n(\vW)}{1-\lambda_2(\vW)}$ are typical in the literature of optimization and average consensus, respectively. 
The value $\frac{L}{\mu}$, also called the condition number of the objective function, is aligned with the scalability of the standard gradient descent~\cite{nesterov2013introductory}. 
The value $\frac{1-\lambda_n(\vW)}{1-\lambda_2(\vW)}$ is understood as the condition number\footnote{When we choose $\vW=\vI-\tau\textbf{\L}$ where $\textbf{\L}$ is the Laplacian of the underlying graph and $\tau$ is a positive tunable constant, we have $\frac{1-\lambda_n(\vW)}{1-\lambda_2(\vW)}=\frac{\lambda_1(\textbf{\L})}{\lambda_{n-1}(\textbf{\L})}$ which is the finite condition number of $\textbf{\L}$. Note that $\lambda_{n}(\textbf{\L})=0$.} of the network and aligned with the scalability of the simplest linear iterations for distributed averaging~\cite{nedic2009distributed}.

Separating the condition numbers of the objective function and the network provides a way to determine the bottleneck of NIDS for a specific problem and a given network. 
Therefore, the system designer might be able to smartly apply preconditioning on $\{s_i\}_{i=1}^n$ or improve the connectivity of the network to cost-effectively obtain a better convergence. 

\begin{table*}[!t]
\centering
	\caption{Summary of a few relevant algorithms. 
	Here, $\mu$ (or $\mu_g$ or $\bar{\mu}$) is the strong convexity constant of the objective function (or that of a modified objective function); $L$ (or $\bar{L}$) is the Lipschitz constant of the objective gradient (or its modified version). 
	Also $\sigma=\frac{1-\lambda_n(\vW)}{1-\lambda_2(\vW)}$ is considered as the condition number of the network, which scales at the order of $O(n^2)$ in the worst case scenario. 
	We omit ``$O(\cdot)$'' in ``Orders of step-size bounds'' and ``Scalability'' for brevity. 
	Note quantities involving $K$ only hold for a  finite $K$.}\label{tab: sumup}
\begin{tabular}{c|c|c|c|c|c} 
		\hline\hline
		Algorithm &Support &Orders of step-size bounds& Uncoordinated&Scalability&Rate\\
		\ &prox. operators&(strongly convex; convex)&step-size&(strongly convex)&(convex) \\
		\hline\hline
EXTRA&&&&&\\		\cite{Shi2015,shi2015proximal}&Yes&$\frac{\mu_g}{L^2}$; $\frac{1}{L}$&No&$\left(\frac{L}{\mu_g}\right)^2$&$o\left(\frac{1}{k}\right)$\\
		\hline
        Aug-DGM&&&&&\\
		\cite{Xu2015}&No&small enough&Yes& -- & converges\\
		\hline
Harness&&&&&\\	\cite{Qu2016}&No&$\frac{\mu}{L^2\sigma^2}$; $\frac{1}{L\sigma^2}$&No&$\left(\frac{L}{\mu}\sigma\right)^2$&$O\left(\frac{1}{k}\right)$\\
		\hline
Acc-DNGD&&$\frac{(\sigma-1)^3}{\sigma^6 L}\left(\frac{\mu}{L}\right)^{3/7}$;&&&\\		\cite{qu2017accelerated}&No& $\frac{\min\{(1-\sigma^{-1})^2,(\sigma)^{-3}\}}{Lk^{0.6}}$&No&$\frac{\sigma^3}{(\sigma-1)^{1.5}}\left(\frac{L}{\mu}\right)^{5/7}$&$O\left(\frac{1}{k^{1.4}}\right)$\\
		\hline
DIGing&&&&&\\			\cite{nedich2016achieving,nedic2016geometrically}&No&$\min\{\frac{\bar{\mu}^{0.5}}{\sigma(\sigma-1)L^{1.5}n^{0.5}},\frac{1}{\bar{L}}\}$; --&Yes&$\max\{\frac{L}{\bar{\mu}},\frac{(\sigma-1)^2n^{0.5}L^{1.5}+\sigma^2\bar{\mu}^{1.5}}{\bar{\mu}^{1.5}}\}$&-- \\
		\hline
Optimal&&&&&\\		\cite{scaman2017optimal,uribe2017optimal}&No&$\mu$; $\frac{L\sigma}{K^2}$&No&$\left(\frac{L}{\mu}\sigma\right)^{0.5}$&$O\left(\frac{1}{K^2}\right)$\\
		\hline
		NIDS&Yes&$\frac{1}{L}$; $\frac{1}{L}$&Yes&$\max\{\frac{L}{\mu},\sigma\}$&$o(\frac{1}{k})$\\
		\hline
		\hline
	\end{tabular}
\end{table*}

{\bf Summary and comparison of state-of-the-art algorithms:}
We list the properties of a few relevant algorithms in Table~\ref{tab: sumup}. 
We let $\sigma:=\frac{1-\lambda_n(\vW)}{1-\lambda_2(\vW)}$. 
This quantity is directly affected by the network topology and how the matrix $\vW$ is defined, thus is also related to the consensus ability of a network. 
When the network is fully connected (a complete graph), we can choose $\vW$ so that $\lambda_2(\vW)=\lambda_n(\vW)=0$ and thus $\sigma=1$ (the best case); in general $\sigma\geq1$ since $0<1-\lambda_2(\vW)\leq 1-\lambda_n(\vW)<2$; in the worst case, we have $\sigma\leq\frac{1}{1-\lambda_2(\vW)}=O(n^2)$~\cite[Section 2.3]{Olshevsky2014}. 
We keep $\sigma$ in the bounds/rates of involved algorithms for a fair comparison instead of focusing on the worst case that often gives pessimistic/conservative results. 
We omit ``$O(\cdot)$'' in ``Bounds of step-sizes'' and ``Scalabilities'' for brevity and only compare the effect of functional properties ($\mu$ and $L$) and network properties ($\sigma$ and/or $n$).
Before talking into details, let us clarify a few points. 
In EXTRA, $\mu_g$ is a quantity that is associated with the strong convexity of the original function $\bar{f}(x)$, so it covers a larger class of problems; 
In DIGing, $\bar{\mu}$ is the mean value of the strong convexity constants of local objectives; 
In Acc-DNGD, the step-size for the convex case contains $k$, the current number of iterations. Thus it represents a diminishing step-size sequence; 
In Optimal~\cite{scaman2017optimal,uribe2017optimal}, the total number of iterations $K$ is used to determine the step-size for the convex case. 
In addition, they apply to problems in which the objectives are dual friendly (see~\cite{uribe2017optimal} for its definition). 
Note some types of objectives are suitable for gradient update, some are suitable for dual gradient update (dual friendly), and some are suitable for proximal update. 

Finding the algorithm with the lowest per-iteration cost depends on the problem (functions). 
Apparently, our bounds on step-sizes and the corresponding scalability/rate are better than those given in EXTRA and Harness (see Table~\ref{tab: sumup}). 
When $\sigma$ is close to $1$ (the graph is well connected), the step-size bound and scalability given in DIGing are the same as NIDS. 
However, when $\sigma$ is large, their result becomes rather conservative. 
Acc-DNGD and Optimal have improved the scalability/rate of gradient-based distributed optimization by employing Nesterov's acceleration technique on primal and dual problems, respectively. 
For the convex case, our rate is worse than theirs because our algorithm does not employ Nesterov's acceleration. 
For the primal distributed gradient method after acceleration~\cite{qu2017accelerated}, the scalability in $\sigma$ is still worse than our result. 
Algorithm Optimal achieves the optimal scalability/rate for distributed optimization. 
However, as we have mentioned above, their algorithms are dual based thus apply to a different class of problems. 
In addition, NIDS supports proximable non-smooth functions and uncoordinated step-sizes while these have not been considered in Acc-DNGD and Optimal. 
To sum up, we have reached the best possible performance of first-order algorithms for distributed optimization without acceleration. 
Further improving the performance by incorporating Nesterov's techniques to our algorithm will be a future direction.

Finally, we note that, references~\cite{yuan2017exact1,yuan2017exact2}, appearing simultaneously with this work, also proposed~\eqref{eq:compare1b} to enlarge the step-size and use column stochastic matrices rather than symmetric doubly stochastic matrices. 
However, their algorithm only works for smooth problems, and their analysis seems to be restrictive and requires twice differentiability and strong convexity of $\{s_i\}_{i=1}^n$. 
The stepsize is also in the order of $\mu/L^2$~\cite{2019arXiv190401196A}.


\subsection{Future Works}
The capability of our algorithm using purely locally determined parameters increases its potential to be extended to dynamic networks with a time-varying number of nodes. 
Given such flexibility, we may use similar schemes to solve the decentralized empirical risk minimization problems. 
Furthermore, it also enhances the privacy of the agents through allowing each agent to perform their own optimization procedure without negotiation on any parameter.

By using Nesterov's acceleration technique, reference~\cite{Olshevsky2014} shows that the scalability of a new average consensus protocol can be improved to $O(n)$; when the nonsmooth terms $r_i$'s are absent, reference~\cite{scaman2017optimal} shows that the scalability of a new dual based accelerated distributed gradient method can be improved to $O(\sqrt{\sigma L/\mu})$. One of our future work is exploring the convergence rates/scalability of the Nesterov's accelerated version of our algorithm.

\subsection{Paper Organization}
The rest of this paper is organized as follows. 
To facilitate the description of the technical ideas, the algorithms, and the analysis, we introduce additional notation in Subsection~\ref{sec:notation}. 
The intuition for the network-independent step-size is provided in Section~\ref{sec:intuition}. 
In Section~\ref{sec:alg}, we introduce our algorithm NIDS and discuss its relation to some other existing algorithms. 
In Section~\ref{sec:conv_analysis}, we first show that NIDS can be understood as an iterative algorithm for seeking a fixed point. 
Following this, we establish that NIDS converges at an $o(1/k)$ rate for the general convex case and a linear rate for the strongly convex case. 
Then, numerical simulations are given in Section~\ref{sec:num_exp} to corroborate our theoretical claims. 
Final remarks are given in Section~\ref{sec:concl}.

\subsection{Notation}\label{sec:notation}
We use bold upper-case letters such as $\vW$ to define matrices in $\RR^{n\times n}$ and bold lower-case letters such as $\vx$ and $\vz$ to define matrices in $\RR^{n\times p}$ (when $p=1$, they are vectors). 
Let $\vone$ and $\vzero$ be matrices with all ones and zeros, respectively, and their dimensions are provided when necessary.  
For matrices $\vx,~\vy\in\RR^{n\times p}$, we define their inner product as $\langle \vx,\vy\rangle=\tr(\vx^\top\vy)$ and the norm as $\|\vx\|=\sqrt{\langle \vx,\vx\rangle}$. 
Additionally, by an abuse of notation, we define $\langle \vx,\vy\rangle_\vQ = \tr(\vx^\top\vQ\vy)$ and $\|\vx\|_\vQ^2=\langle \vx,\vx\rangle_\vQ$ for any given symmetric matrix $\vQ\in\RR^{n\times n}$. 
Note that $\langle\cdot,\cdot\rangle_\vQ$ is an inner product defined in $\RR^{n\times p}$ if and only if $\vQ$ is positive definite. 
However, when $\vQ$ is not positive definite, $\langle\vx,\vy\rangle_\vQ$ can still be an inner product defined in a subspace of $\RR^{n\times p}$, see Lemma~\ref{lemma:equiv_norm} for more details. We define the range of $\vA\in \RR^{n\times n}$ by $\range(\vA):=\{\vx\in \RR^{n\times p}:\vx=\vA\vy,~\vy\in\RR^{n\times p}\}$. The largest eigenvalue of a symmetric matrix $\vA$ is also denoted as $\lambda_{\max}(\vA)$. For two symmetric matrices $\vA,\vB\in\RR^{n\times n}$, $\vA\succ\vB$ (or $\vA\succcurlyeq\vB$) means that $\vA-\vB$ is positive definite (or positive semidefinite). Moreover, we use $\cN_i$ to represent the set of agents that can directly send messages to agent~$i$. 


\section{Intuition for Network-independent step-size}\label{sec:intuition}
In this section, we provide an intuition for the network-independent step-size for NIDS with only the differentiable function $s$.
The decentralized optimization problem is equivalent to 
$$\Min_\vx~s(\vx), \quad \st\quad (\vI-\vW)^{1/2}\vx=\vzero,$$
where $(\vI-\vW)^{1/2}$ is the square root of $\vI-\vW$, and the constraint is the same as the consensual condition with the mixing matrix $\vW$ given in Assumption~\ref{assu:1}. 
Denote $\Lap = (\vI-\vW)^{1/2}$.
The corresponding optimality condition with the introduction of the dual variable $\vp$ is 
\begin{align*}
    \left[\begin{array}{cc} 0 & \Lap\\ -\Lap& 0\end{array}\right]    \left[\begin{array}{c} \vx^*\\ \vp^* \end{array}\right]=  - \left[\begin{array}{c}\nabla s(\vx^*) \\ 0\end{array}\right].
\end{align*}
EXTRA is equivalent to the Condat-Vu primal-dual algorithm~\cite{wu2018decentralized,li2017aprimal}, and it can be further explained as a forward-backward splitting applied to the equation, i.e.,
\begin{align*}
&\left[\left[\begin{array}{cc} {1\over \alpha}\vI & -\Lap\\ -\Lap& 2\alpha\vI\end{array}\right] +     \left[\begin{array}{cc} 0 & \Lap\\ -\Lap& 0\end{array}\right]   \right] \left[\begin{array}{c} \vx^{k+1}\\ \vp^{k+1} \end{array}\right] \\
&=  \left[\begin{array}{cc} {1\over\alpha}\vI & -\Lap\\ -\Lap& 2\alpha\vI\end{array}\right] \left[\begin{array}{c} \vx^{k}\\ \vp^{k} \end{array}\right] -\left[\begin{array}{c}\nabla s(\vx^k) \\ 0\end{array}\right].
\end{align*}
The update is 
\begin{align*}
    \vx^{k+1} = & \vx^k-\alpha\Lap\vp^k - \alpha \nabla s(\vx^k),\\
    \alpha\vp^{k+1} = & \alpha\vp^k-{1\over 2}\Lap\vx^k+\Lap\vx^{k+1}.
\end{align*}
It is equivalent to EXTRA after $\vp$ is eliminated.
In this case, the new metric is a full matrix, and therefore, the upper bound of the step-size $\alpha$ depends on the matrix $\Lap$. 
To be more specific, $$\left[\begin{array}{cc} {1\over\alpha}\vI & -\Lap\\ -\Lap& 2\alpha\vI\end{array}\right]\succcurlyeq \left[\begin{array}{cc} {L\over 2}\vI & \vzero\\ \vzero& \vzero\end{array}\right],$$ which gives  $\alpha \leq 2(1+\lambda_{\min}(\vW))/L$. 
A larger and optimal upper bound for the step-size of  EXTRA is shown in~\cite{li2017aprimal} (See Table~\ref{tab:comp_EXTRA_NIDS}), and it still depends on $\vW$.
However, we choose a block diagonal metric and have 
\begin{align*}
&\left[\left[\begin{array}{cc} {1\over\alpha}\vI & 0\\ 0& \alpha(\vI+\vW)\end{array}\right] +     \left[\begin{array}{cc} 0 & \Lap\\ -\Lap& 0\end{array}\right]   \right] \left[\begin{array}{c} \vx^{k+1}\\ \vp^{k+1} \end{array}\right] \\
=&  \left[\begin{array}{cc} {1\over\alpha}\vI & 0\\ 0& \alpha(\vI+\vW)\end{array}\right] \left[\begin{array}{c} \vx^{k}\\ \vp^{k} \end{array}\right] -\left[\begin{array}{c}\nabla s(\vx^k) \\ 0\end{array}\right].
\end{align*}
The update becomes 
\begin{align*}
    2\alpha\vp^{k+1} = & \alpha(\vI+\vW)\vp^k+\Lap\vx^k-\alpha\Lap\nabla s(\vx^k),\\
    \vx^{k+1} = & \vx^k - \alpha \nabla s(\vx^k)-\alpha \Lap\vp^{k+1},
\end{align*}
which is equivalent to NIDS after $\vp$ is eliminated.
Because the new metric is block diagonal, and the nonexpansiveness of the forward step depends on the function only, i.e., $\alpha\leq 2/L$.

\section{Proposed Algorithm NIDS}\label{sec:alg}
In this section, we describe our proposed NIDS in Algorithm~\ref{alg:alg1} for solving~\eqref{eq:F} in more details and explain the connections to other related methods.
\begin{algorithm}[!ht]
	\caption{NIDS}
	\begin{algorithmic}\label{alg:alg1}
		\STATE Each agent $i$ obtains its mixing values $w_{ij}, \forall j\in\cN_i$;
		\STATE Each agent $i$ chooses its own step-size $\alpha_i>0$ and the same parameter $c$ (e.g., $c=0.5/\max_i\alpha_i$);
		\STATE Each agent $i$ sets the mixing values
		$\widetilde{w}_{ij}:=c\alpha_i w_{ij},\forall j\in\cN_i$ and $\widetilde{w}_{ii}:=1-c\alpha_i+c\alpha_i w_{ii}$;
		\STATE Each agent $i$ picks arbitrary initial $x_i^0 \in \RR^{p}$ and performs
			\begin{align*}
			z_i^{1} & = x_i^{0}-\alpha_i\nabla s_i\left(x_i^{0}\right), \\
			x_i^{1} & = \argmin\limits_{x\in\RR^p}~\alpha_ir_i(x)+\frac{1}{2}\|x-z_i^{1}\|^{2}. 
			\end{align*}
		\FOR {$k=1,2,3\ldots$} 
		  \STATE Each agents $i$ performs	
			\begin{align*}
			z_i^{k+1}  =  z_i^{k}-x_i^k+\sum_{j=\cN_i\cup\{i\}}\widetilde w_{ij} \left(2x_j^{k}-x_j^{k-1}\right. \\ \left. -\alpha_j\nabla s_j\left(x_j^{k}\right)+\alpha_j\nabla s_j\left(x_j^{k-1}\right)\right), \\
			x_i^{k+1}  = \argmin\limits_{x\in\RR^p}~\alpha_i r_i(x)+\frac{1}{2}\|x-z_i^{k+1}\|^{2}.
			\end{align*}
		\ENDFOR
	\end{algorithmic}
\end{algorithm}

The mixing matrix satisfies the following assumption, which comes from~\cite{shi2015proximal,Shi2015}.
\begin{assumption}[Mixing matrix]\label{assu:1} 
The connected network $\cG=\{\cV,\cE\}$ consists of a set of agents $\cV=\{1,2,\cdots,n\}$ and a set of undirected edges $\cE$. 
An undirected edge $(i,j)\in\cE$ means that there is a connection between agents $i$ and $j$ and both agents can exchange data.
The mixing matrix $\vW=[w_{ij}]\in \RR^{n\times n}$ satisfies:
\begin{enumerate}
\item \textnormal{(Decentralized property).} If $i\neq j$ and $(i,j)\notin \cE$, then $w_{ij}=0$;
\item	\textnormal{(Symmetry).} $\vW = \vW^T$;
\item	\textnormal{(Null space property).} $\Null(\vI-\vW)=\Span(\vone_{n\times 1})$;
\item \textnormal{(Spectral property).} $2\vI \succcurlyeq \vW+\vI\succ \vzero_{n\times n}$.
\end{enumerate}
\end{assumption}

\begin{remark}\label{remark:W}
Assumption~\ref{assu:1} implies that the eigenvalues of $\vW$ lie in $(-1,1]$ and the multiplicity of eigenvalue $1$ is one, i.e., $1=\lambda_1(\vW)>\lambda_2(\vW)\geq \cdots \geq \lambda_n(\vW)>-1$. 
Item 3 of Assumption~\ref{assu:1} shows that $(\vI-\vW)\vone_{n\times 1} = \vzero$ and the orthogonal complement of $\Span(\vone_{n\times 1})$ is the row space of $\vI-\vW$, which is also the column space of $\vI-\vW$ because of the symmetry of $\vW$.
\end{remark}

The functions $\{s_i\}_{i=1}^n$ and $\{r_i\}_{i=1}^n$ satisfy the following assumption.
\begin{assumption} \label{lemma:cocoerciveness}
Functions $\{s_i(x)\}_{i=1}^n$ and $\{r_i(x)\}_{i=1}^n$ are lower semi-continuous proper convex, and $\{s_i(x)\}_{i=1}^n$ have Lipschitz continuous gradients with constants $\{L_i\}_{i=1}^n$, respectively. Thus, we have
\begin{equation}\label{for:cocoer}
\langle \vx-\vy,\nabla s(\vx)-\nabla s(\vy)\rangle\geq \left\| \nabla s\left(\vx\right)-\nabla s\left(\vy\right)\right\|_{\vL^{-1}}^{2},
\end{equation}
where $\vL=\Diag(L_1,\cdots,L_n)$ is the diagonal matrix with the Lipschitz constants~\cite{nesterov2013introductory}. 
\end{assumption}

Instead of using the same step-size for all the agents, we allow agent $i$ to choose its own step-size $\alpha_i$ and let $\Lambda=\Diag(\alpha_1,\cdots,\alpha_n)\in \RR^{n\times n}$. 
Then NIDS can be expressed as 
\begin{subequations}\label{for:alg1}
\begin{align}
\begin{split}
\vz^{k+1} =&  \vz^{k}-\vx^k +  \widetilde \vW(2\vx^k-\vx^{k-1} \\ 
&-\Lambda\nabla s(\vx^k)+\Lambda\nabla s(\vx^{k-1})),
\end{split}\label{for:alg1_z}\\
	\vx^{k+1} = &  \argmin\limits_{\vx\in\RR^{n\times p}}~r(\vx)+\frac{1}{2}\|\vx-\vz^{k+1}\|_{\Lambda^{-1}}^{2}, \label{for:alg1_x}
	\end{align}
\end{subequations}
where $\widetilde\vW = \vI-c\Lambda(\vI-\vW)$ and $c$ is chosen such that $\Lambda^{-1/2}\widetilde\vW\Lambda^{1/2}=\vI-c\Lambda^{1/2}(\vI-\vW)\Lambda^{1/2}\succcurlyeq \vzero$. 
This condition shows that the upper bound of the parameter $c$ depends on $\vW$ and $\Lambda$. 
When the information about $\vW$ is not given, we can just let $c=1/(2\max_i\alpha_i)$ because $\lambda_n(\vW)>-1$. 
To set such a parameter, a preprocessing step is needed to obtain the maximum. However, since the maximum can be easily computed in a connected network in no more than $n-1$ rounds of communication wherein each node repeatedly takes maximum of the values from neighbors, the cost of this preprocessing is essentially negligible compared to the worst-case running time of our optimization protocol.


If all agents choose the same step-size, i.e., $\Lambda=\alpha\vI$, and we let $c=1/(2\alpha)$, ~\eqref{for:alg1} becomes
\begin{subequations}
\begin{align}
\begin{split}
	\vz^{k+1}  = & \vz^{k}-\vx^k+{\vI+\vW\over 2}(2\vx^k-\vx^{k-1}\\
    &-\alpha\nabla s(\vx^k)+\alpha\nabla s(\vx^{k-1})),
\end{split}\\
	\vx^{k+1}  = & \argmin\limits_{\vx\in\RR^{n\times p}}~r(\vx)+\frac{1}{2\alpha}\|\vx-\vz^{k+1}\|^{2}.
	\end{align}
\end{subequations}

\begin{remark} 
The update of PG-EXTRA is
\begin{subequations}\label{pg-extra}
\begin{align}
\begin{split}
	\vz^{k+1}  = & \vz^{k}-\vx^k+{\vI+\vW\over 2}(2\vx^k-\vx^{k-1})\\
    &-\alpha\nabla s(\vx^k)+\alpha\nabla s(\vx^{k-1}),
\end{split}\\
	\vx^{k+1}  = & \argmin\limits_{\vx\in\RR^{n\times p}}~r(\vx)+\frac{1}{2\alpha}\|\vx-\vz^{k+1}\|^{2}.
	\end{align}
\end{subequations}
The only difference between NIDS and PG-EXTRA is that the mixing operation is further applied to the successive difference of the gradients $-\alpha\nabla s(\vx^k)+\alpha\nabla s(\vx^{k-1})$ in NIDS.
\end{remark}

When there is no function $r(\vx)$,~\eqref{for:alg1} becomes 
\begin{align*}
	\vx^{k+1} = &\widetilde \vW(2\vx^k-\vx^{k-1} -\Lambda\nabla s(\vx^k)+\Lambda\nabla s(\vx^{k-1})),
\end{align*}
and it further reduces to~\eqref{eq:compare1b} when $\Lambda=\alpha\vI$ and $c=1/(2\alpha)$. 
Note that, though~\eqref{eq:compare1b} appears in~\cite{yuan2017exact1,yuan2017exact2}, its convergence still needs a small step-size that also depends on the network topology and the strong convexity constant. 
In Theorem 1 of~\cite{yuan2017exact2}, the upper bound for the step-size is also $O(\mu/L^2)$, which is the same as that of PG-EXTRA.


\section{Convergence Analysis of NIDS}\label{sec:conv_analysis}
In order to show the convergence of NIDS, we also need the following assumption.
\begin{assumption}[Solution existence]
Problem~\eqref{eq:F} has at least one solution.
\end{assumption}

To simplify the analysis, we introduce a new sequence $\{\vd^{k}\}_{k\geq0}$ which is defined as 
\begin{align}\label{def:d}
\vd^{k}:={\Lambda^{-1}}(\vx^{k-1}-\vz^k)-\nabla s\left(\vx^{k-1}\right).
\end{align} 
Using the sequence $\{\vx^{k}\}_{k\geq0}$, we obtain a recursive  (update) relation for $\{\vd^{k}\}_{k\geq0}$:
\begin{align*}
    &  \vd^{k+1}\\
=		& {\Lambda^{-1}}(\vx^{k}-\vz^{k+1})-\nabla s(\vx^{k}) \\
=		& {\Lambda^{-1}}(\vx^{k}-\vz^k+\vx^k)-\nabla s(\vx^k)\\
    & -{\Lambda^{-1}}\widetilde\vW(2\vx^k-\vx^{k-1} 
     -\Lambda\nabla s(\vx^k)+\Lambda\nabla s(\vx^{k-1})) \\
=		& {\Lambda^{-1}}(\vx^{k}-\vz^k+\vx^k-2\vx^k+\vx^{k-1}) \\
&-\nabla s(\vx^{k})+\nabla s(\vx^k)-\nabla s(\vx^{k-1}) \\
    & +c(\vI-\vW)(2\vx^k-\vx^{k-1}-\Lambda\nabla s(\vx^k)+\Lambda\nabla s(\vx^{k-1})) \\
=		& \vd^k+c(\vI-\vW)(2\vx^k-\vz^{k}-\Lambda\nabla s(\vx^k)-\Lambda\vd^k),
\end{align*}
where the second equality comes from the update of $\vz^{k+1}$ in~\eqref{for:alg1_z} and the last one holds because of the definition of $\vd^k$ in~\eqref{def:d}.
Therefore, the iteration~\eqref{for:alg1} is equivalent to, with the update order $(\vx,\vd,\vz)$,
\begin{subequations}\label{for:alg2}
\begin{align}
\vx^{k}     = &\argmin\limits_{\vx\in\RR^{n\times p}}~r(\vx)+\frac{1}{2 }\|\vx-\vz^{k}\|_{\Lambda^{-1}}^{2}, \label{for:alg2_x} \\
\begin{split}
\vd^{k+1}  = &\vd^k + c(\vI-\vW)\left(2\vx^{k}-\vz^{k} \right. \\ &\left. -\Lambda\nabla s\left(\vx^{k}\right)-\Lambda\vd^k\right),
\end{split} \label{for:alg2_d}\\
\vz^{k+1}  = & \vx^{k}-\Lambda\nabla s\left(\vx^{k}\right)-\Lambda\vd^{k+1}, \label{for:alg2_z} 
\end{align}
\end{subequations}
in the sense that both \eqref{for:alg1} and \eqref{for:alg2} generate the same $\{\vx^k,\vz^k\}_{k>0}$ sequence. 


Because $\vx^k$ is determined by $\vz^k$ only and can be eliminated from the iteration, iteration~\eqref{for:alg2} is essentially an operator for $(\vd,\vz)$.
Note that we have $\vd^1={\Lambda^{-1}}(\vx^{0}-\vz^1)-\nabla s\left(\vx^{0}\right)=\vzero$ from Algorithm~\ref{alg:alg1}. 
Therefore, from the update of $\vd^{k+1}$ in~\eqref{for:alg2_d}, $\vd^k\in\range(\vI-\vW)$ for all $k$. 
In fact, any $\vz^1$ such that $\vd^1\in\range(\vI-\vW)$ works for NIDS.
The following two lemmas show the relation between fixed points of~\eqref{for:alg2} and optimal solutions of~\eqref{eq:F}. 
The proofs for all lemmas and propositions are included in the supplemental material. 

\begin{lemma}[Fixed point of~\eqref{for:alg2}]\label{lemma:fixedpoint}
$(\vd^*,\vz^*)$ is a fixed point of~\eqref{for:alg2} if and only if there exists a subgradient $\vq^*\in\partial r(\vx^*)$ such that $\vz^*=\vx^*+\Lambda\vq^*$ and
\begin{subequations}
\begin{align}
&\vd^* + \nabla s(\vx^*)+\vq^* & = \vzero ,\label{for:fpopt1}\\
&(\vI-\vW)\vx^* &=\vzero.\label{for:fpopt2}
\end{align}
\end{subequations}
\end{lemma}



\begin{lemma}[Optimality condition]\label{lemma:condition}
$\vx^*$ is consensual with $x^*_1=x^*_2=\cdots=x^*_n=x^*$ being an optimal solution of problem~\eqref{eq:F} if and only if there exists $\vp^*$ and a subgradient $\vq^*\in\partial r(\vx^*)$  such that:
\begin{subequations}\label{for:opt}
\begin{align}
&(\vI-\vW)\vp^* + \nabla s(\vx^*)+\vq^* &= \vzero ,\label{for:opt1}\\
&(\vI-\vW)\vx^*&=\vzero .\label{for:opt2}
\end{align}
\end{subequations}
In addition, $(\vd^*=(\vI-\vW)\vp^*, \vz^*=\vx^*+\Lambda\vq^*)$ is a fixed point of iteration~\eqref{for:alg2}.
\end{lemma}

Lemma~\ref{lemma:condition} shows that we can find a fixed point of iteration~\eqref{for:alg2} to obtain an optimal solution of problem~\eqref{eq:F}. 
It also tells us that we need $\vd^*\in\range(\vI-\vW)$ to get an optimal solution of problem~\eqref{eq:F}. 
Therefore, we need $\vd^1\in\range(\vI-\vW)$.




\begin{lemma}[Norm over range space]\label{lemma:equiv_norm}
For any symmetric positive semidefinite matrix $\vA\in \RR^{n\times n}$ with rank $r\leq n$, 
let $\lambda_1\geq \lambda_2\geq \dots\geq \lambda_r>0$ be its $r$ eigenvalues. 
Then $\range(\vA)$ defined in Section~\ref{sec:notation} is a $rp$-dimensional subspace in $\RR^{n\times p}$ and has a norm defined by $\|\vx\|^2_{\vA^\dagger}:=\langle \vx, \vA^\dagger \vx\rangle$, where $\vA^\dagger$ is the pseudo inverse of $\vA$. 
In addition, $\lambda_1^{-1}\|\vx\|^2\leq \|\vx\|^2_{\vA^\dagger}\leq \lambda_r^{-1}\|\vx\|^2$ for all $\vx\in\range(\vA)$.
\end{lemma}



\begin{proposition} \label{prop:M}
Let $\vM=c^{-1}(\vI-\vW)^\dagger-\Lambda$ with $\vI\succcurlyeq c\Lambda^{1/2}(\vI-\vW)\Lambda^{1/2} \succcurlyeq\vzero$. Then $\|\cdot\|_{\vM}$ is a norm defined for $\range(\vI-\vW)$.
\end{proposition}

The following lemma compares the distance to a fixed point of~\eqref{for:alg2} for two consecutive iterates. 

\begin{lemma}[Fundamental inequality]\label{lemma:powerIneq}
Let $(\vd^*,\vz^*)$ be a fixed point of iteration~\eqref{for:alg2} with $\vd^*\in\range(\vI-\vW)$. The update $(\vd^k,\vz^k)\Rightarrow(\vd^{k+1},\vz^{k+1})$ in~\eqref{for:alg2} satisfies
\begin{align}
&\|\vz^{k+1}-\vz^*\|^2_{\Lambda^{-1}}+\|\vd^{k+1}-\vd^*\|_\vM^2 \nonumber\\
\begin{split}
\leq		& \|\vz^{k}-\vz^*\|^2_{\Lambda^{-1}}+\|\vd^{k}-\vd^*\|_\vM^2 \\
& -\|\vz^{k}-\vz^{k+1}\|^2_{\Lambda^{-1}}-\|\vd^{k}-\vd^{k+1}\|_\vM^2  \\
&  +2\langle\nabla s\left(\vx^{k}\right)-\nabla s\left(\vx^{*}\right),\vz^{k}-\vz^{k+1}\rangle \\
& -2\langle \vx^{k}-\vx^{*},\nabla s(\vx^{k})-\nabla s(\vx^{*})\rangle.
\end{split}\label{for:powerIneq}
\end{align}
\end{lemma}

\begin{IEEEproof} 
From the update of $\vz^{k+1}$ in~\eqref{for:alg2_z}, we have
\begin{align}
 & \langle \vd^{k+1}-\vd^*,\vz^{k+1}-\vz^k+\vx^k-\vx^*\rangle\nonumber\\
=& \langle \vd^{k+1}-\vd^{*},2\vx^{k}-\vz^{k}-\Lambda\nabla s(\vx^{k})-\Lambda\vd^{k+1}-\vx^*\rangle \nonumber\\
=& \langle \vd^{k+1}-\vd^{*},c^{-1}(\vI-\vW)^\dagger(\vd^{k+1}-\vd^k)\nonumber \\
&+\Lambda\vd^{k}-\Lambda\vd^{k+1}\rangle\nonumber\\
=& \langle \vd^{k+1}-\vd^*,\vd^{k+1}-\vd^k\rangle_\vM, \label{for:main_ineq_1}
\end{align}
where the second equality comes from~\eqref{for:alg2_d},~\eqref{for:opt2}, and $\vd^{k+1}-\vd^*\in \range(\vI-\vW)$.
From~\eqref{for:alg2_x}, we have that 
\begin{align}\label{for:proximal_ineq}
\langle \vx^k-\vx^*, \vz^k-\vx^k-\vz^*+\vx^*\rangle_{\Lambda^{-1}} \geq 0.
\end{align}	
Therefore, we have
\begin{align} 
  & \langle \vx^k-\vx^*,\nabla s(\vx^k)-\nabla s(\vx^*)\rangle\nonumber \\
\leq& \langle \vx^k-\vx^*,{\Lambda^{-1}}(\vz^k-\vx^k-\vz^*+\vx^*)\nonumber \\
&+\nabla s(\vx^k)- \nabla s(\vx^*)\rangle  \nonumber\\
= & \langle \vx^k-\vx^*,{\Lambda^{-1}}(\vz^k-\vz^{k+1})-\vd^{k+1}+\vd^*\rangle\nonumber\\
= & \langle \vx^k-\vx^*,\vz^k-\vz^{k+1}\rangle_{\Lambda^{-1}} + \langle \vd^{k+1}\nonumber\\
&-\vd^*,\vz^{k+1}-\vz^{k}\rangle - \langle \vd^{k+1}-\vd^{*},\vd^{k+1}-\vd^k\rangle_\vM  \nonumber\\
= & \langle {\Lambda^{-1}}(\vx^k-\vx^*)-\vd^{k+1}+\vd^*,\vz^k-\vz^{k+1}\rangle  \nonumber\\
&-\langle \vd^{k+1}-\vd^{*},\vd^{k+1}-\vd^k\rangle_\vM  \nonumber\\
= & \langle {\Lambda^{-1}}(\vz^{k+1}-\vz^*)+\nabla s(\vx^k)-\nabla s(\vx^*),\vz^k-\vz^{k+1}\rangle\nonumber\\
& - \langle \vd^{k+1}-\vd^{*},\vd^{k+1}-\vd^k\rangle_\vM  \nonumber\\
= & \langle \vz^{k+1}-\vz^*,\vz^k-\vz^{k+1}\rangle_{\Lambda^{-1}} \nonumber \\
&+ \langle \nabla s(\vx^k)-\nabla s(\vx^*),\vz^k-\vz^{k+1}\rangle\nonumber \\
  & + \langle \vd^{k+1}-\vd^*,\vd^{k}-\vd^{k+1}\rangle_\vM  \nonumber.
\end{align}
The inequality and the second equality comes from~\eqref{for:proximal_ineq} and~\eqref{for:main_ineq_1}, respectively. The first and fourth equalities hold because of the update of $\vz^{k+1}$ in~\eqref{for:alg2_z}. Using $2\langle \va,\vb\rangle=\|\va+\vb\|^{2}-\|\va\|^{2}-\|\vb\|^{2}$ and rearranging the previous inequality give us that
\begin{align*}
& 2\langle \vx^{k}-\vx^{*},\nabla s(\vx^{k})-\nabla s(\vx^{*})\rangle \\
&- 2\langle\nabla s(\vx^{k})-\nabla s(\vx^{*}),\vz^{k}-\vz^{k+1}\rangle \nonumber\\
\leq		&2\langle\vz^{k+1}-\vz^{*},\vz^{k}-\vz^{k+1}\rangle_{\Lambda^{-1}} \nonumber \\
&+2\langle \vd^{k+1}-\vd^{*},\vd^{k}-\vd^{k+1}\rangle_{\vM} \nonumber\\
=		&\|\vz^{k}-\vz^{*}\|^{2}_{\Lambda^{-1}}-\|\vz^{k+1}-\vz^{*}\|^{2}_{\Lambda^{-1}}-\|\vz^{k}-\vz^{k+1}\|^{2}_{\Lambda^{-1}} \nonumber\\
&  +\|\vd^{k}-\vd^{*}\|_{\vM}^{2}-\|\vd^{k+1}-\vd^{*}\|_{\vM}^{2}-\|\vd^{k}-\vd^{k+1}\|_{\vM}^{2}.
\end{align*}
Therefore,~\eqref{for:powerIneq} is obtained.
\end{IEEEproof}

\subsection{Sublinear convergence of NIDS} 
As explained in Section~\ref{sec:intuition}, NIDS is equivalent to the primal-dual algorithm~\cite{Yan2016} applied to problem 
\begin{align}\label{NIDS-PD3O}
\Min_\vx~s(\vx) + r(\vx) + \iota((\vI-\vW)^{1/2}\vx),
\end{align} 
where $\iota(\cdot)$ is the indicator function, which return 0 for $\vzero$ and $+\infty$ otherwise,
with the metric matrix being 
$$\left[\begin{array}{cc} \Lambda^{-1} & \vzero\\ \vzero & c^{-1}\vI - (\vI-\vW)^{1/2}\Lambda(\vI-\vW)^{1/2}\end{array}\right].$$
We apply~\cite[Theorem 1]{Yan2016} and obtain the following sublinear convergence result.

\begin{theorem}[Sublinear rate]\label{lemma:conv_rate}
Let $(\vd^k,\vz^k)$ be the sequence generated from NIDS in~\eqref{for:alg2} with $\alpha_i< 2/L_i$ for all $i$ and $\vI\succcurlyeq c\Lambda^{1/2}(\vI-\vW)\Lambda^{1/2}$. 
We have
\begin{align}\label{conv:sublinear}
\begin{split}
&\|\vz^{k}-\vz^{k+1}\|^2_{\Lambda^{-1}}+\|\vd^{k}-\vd^{k+1}\|_{\vM}^{2} \\
&\quad  \leq   \frac{\|\vz^{1}-\vz^{*}\|^2_{\Lambda^{-1}} +\|\vd^{1}-\vd^{*}\|_{\vM}^{2} }{k(1-\max_i{\alpha_iL_i\over 2})},\end{split}\\
&\|\vz^{k}-\vz^{k+1}\|^2_{\Lambda^{-1}}+\|\vd^{k}-\vd^{k+1}\|_{\vM}^{2} = & o\left(\frac{1}{k+1}\right).\nonumber
\end{align}
Furthermore, $(\vd^k,\vz^k)$ converges to a fixed point $(\bar\vd,\bar\vz)$ of iteration~\eqref{for:alg2} and $\bar\vd\in\range(\vI-\vW)$, if $\vI\succ c\Lambda^{1/2}(\vI-\vW)\Lambda^{1/2}$.
\end{theorem}

\begin{remark} 
Note the convergence in Theorem~\ref{lemma:conv_rate} is shown in $\vz$ and $\vd$. 
We will show the convergence in terms of~\eqref{for:opt}.
Recall that
\begin{align*}
\vz^{k+1}-\vz^k 
= & \vx^k-\Lambda\nabla s(\vx^k)-\Lambda \vd^{k+1}-\vz^k \\
= &-\Lambda(\vd^{k+1}+\nabla s(\vx^k)+\vq^k),
\end{align*}
where $\vq^k\in\partial r(\vx^k)$. Therefore, $\|\vz^{k+1}-\vz^{k}\|_{\Lambda^{-1}}^{2}\rightarrow0$ implies the convergence in terms of~\eqref{for:opt1}.

Combining~\eqref{for:alg2_d} and~\eqref{for:alg2_z}, we have 
\begin{align*}
\vd^{k+1} 
= & \vd^k + c(\vI-\vW)(\vx^{k}-\vz^{k}+\vz^{k+1})\\
&+c(\vI-\vW)\Lambda(\vd^{k+1}-\vd^k).
\end{align*}
Rearranging it gives 
\begin{align*}
&\left(\vI-c(\vI-\vW)\Lambda\right)\left(\vd^{k+1}-\vd^{k}\right) \\
=&c(\vI-\vW)\left(\vx^{k}-\vz^{k}+\vz^{k+1}\right).
\end{align*}
Then we have 
\begin{align*}
 &\|c(\vI-\vW)\left(\vx^{k}-\vz^{k}+\vz^{k+1}\right)\|^2 \\
= & \|c\left(\vI-\vW\right)\vM^{1/2}\vM^{1/2}\left(\vd^{k+1}-\vd^{k}\right)\|^2 \\
\leq & \| c(\vI-\vW)\vM^{1/2}\|^{2} \left\|\vd^{k+1}-\vd^{k}\right\|^2_\vM,
\end{align*}
where the second equality comes from $\vd^{k+1}-\vd^k\in\range(\vI-\vW)$.
Thus $\|\vz^{k+1}-\vz^{k}\|_{\Lambda^{-1}}^{2}+\|\vd^{k+1}-\vd^k\|^2_\vM\rightarrow0$ implies the convergence in terms of~\eqref{for:opt2}.
\end{remark}


\subsection{Linear convergence for special cases}
In this subsection, we provide the linear convergence rate for the case when $r(\vx)=0$, i.e., $\vz^k=\vx^k$ in NIDS.

\begin{theorem}\label{thm:linear_conv}
If $\{s_i(x)\}_{i=1}^n$ are strongly convex with parameters $\{\mu_i\}_{i=1}^n$, then
\begin{equation}\label{for:strong_conv}
\langle \vx-\vy,\nabla s(\vx)-\nabla s(\vy)\rangle\geq \left\|\vx-\vy\right\|_{\vS}^{2},
\end{equation}
where $\vS=\Diag(\mu_1,\cdots,\mu_n)\in\RR^{n\times n}$. 
Let $(\vd^k,\vx^k)$ be the sequence generated from NIDS with $\alpha_i< 2/L_i$ for all $i$ and $\vI\succcurlyeq c\Lambda^{1/2}(\vI-\vW)\Lambda^{1/2}$. 
We define 
\begin{align}\label{def:rho}
\begin{split}
\rho =\max &\left( 1-(2-\max_i(\alpha_iL_i))\min_i(\mu_i\alpha_i),\right. 
\\  &\left. 1-{c\over \lambda_{\max}(\Lambda^{-1/2}(\vI-\vW)^\dagger\Lambda^{-1/2})} \right),
\end{split}
\end{align}
and have
\begin{align}\label{for:linear_conv}
\begin{split}
&\|\vx^{k+1}-\vx^*\|^2_{\Lambda^{-1}}+\|\vd^{k+1}-\vd^*\|_{\vM+\Lambda}^2  \\ \leq & \rho\left(\|\vx^{k}-\vx^*\|^2_{\Lambda^{-1}}+\|\vd^{k}-\vd^*\|_{\vM+\Lambda}^2\right).\end{split}
\end{align}
\end{theorem}

\begin{IEEEproof}
From~\eqref{for:powerIneq}, we have
\begin{align}
        & \|\vx^{k+1}-\vx^*\|^2_{\Lambda^{-1}}+\|\vd^{k+1}-\vd^*\|_\vM^2 \nonumber\\
\leq		& \|\vx^{k}-\vx^*\|^2_{\Lambda^{-1}}+\|\vd^{k}-\vd^*\|_\vM^2 \nonumber\\
& -\|\vx^{k}-\vx^{k+1}\|^2_{\Lambda^{-1}}-\|\vd^{k}-\vd^{k+1}\|_\vM^2  \label{for:linear_1}\\
        & +2\langle\nabla s\left(\vx^{k}\right)-\nabla s\left(\vx^{*}\right),\vx^{k}-\vx^{k+1}\rangle \nonumber\\
        &-2\langle \vx^{k}-\vx^{*},\nabla s(\vx^{k})-\nabla s(\vx^{*})\rangle.\nonumber
\end{align}
For the two inner product terms, we have
\begin{align}
        & 2\langle\nabla s\left(\vx^{k}\right)-\nabla s\left(\vx^{*}\right),\vx^{k}-\vx^{k+1}\rangle\nonumber\\
        &-2\langle \vx^{k}-\vx^{*},\nabla s(\vx^{k})-\nabla s(\vx^{*})\rangle\cr
=       & -\|\vx^{k}-\vx^{k+1} -\Lambda\nabla s\left(\vx^{k}\right)+\Lambda\nabla s\left(\vx^{*}\right)\|^2_{\Lambda^{-1}}\cr
        & +\|\vx^{k}-\vx^{k+1}\|^2_{\Lambda^{-1}}+\|\nabla s\left(\vx^{k}\right)-\nabla s\left(\vx^{*}\right)\|^2_{\Lambda}\cr
        & -2\langle \vx^{k}-\vx^{*},\nabla s(\vx^{k})-\nabla s(\vx^{*})\rangle\cr
\leq    & -\|\vd^{k+1}-\vd^*\|^2_{\Lambda}+\|\vx^{k}-\vx^{k+1}\|^2_{\Lambda^{-1}}\cr 
&+\|\nabla s\left(\vx^{k}\right)-\nabla s\left(\vx^{*}\right)\|^2_{\Lambda}\cr
        & -\max_i(\alpha_iL_i)\|\nabla s\left(\vx^{k}\right)-\nabla s\left(\vx^{*}\right)\|^2_{\vL^{-1}}\cr 
&-(2-\max_i(\alpha_iL_i))\|\vx^k-\vx^*\|^2_{\vS}\cr
\leq    & -\|\vd^{k+1}-\vd^*\|^2_{\Lambda}+\|\vx^{k}-\vx^{k+1}\|^2_{\Lambda^{-1}} \nonumber\\
        & -(2-\max_i(\alpha_iL_i))\min_i(\mu_i\alpha_i)\|\vx^k-\vx^*\|^2_{\Lambda^{-1}}. \label{for:linear_2}
\end{align}
The first inequality comes from $\vx^{k+1} = \vx^{k} -\Lambda\nabla s\left(\vx^{k}\right)-\vd^{k+1}$, $\Lambda\nabla s\left(\vx^{*}\right)+\vd^*=\vzero$,~\eqref{for:cocoer}, and~\eqref{for:strong_conv}.
Combing~\eqref{for:linear_1} and~\eqref{for:linear_2}, we have
\begin{align*}
          &\|\vx^{k+1}-\vx^*\|^2_{\Lambda^{-1}}+\|\vd^{k+1}-\vd^*\|_\vM^2 \\
\leq		& \|\vx^{k}-\vx^*\|^2_{\Lambda^{-1}}+\|\vd^{k}-\vd^*\|_\vM^2 -\|\vd^{k+1}-\vd^*\|^2_{\Lambda} \\
        & -(2-\max_i(\alpha_iL_i))\min_i(\mu_i\alpha_i)\|\vx^k-\vx^*\|^2_{\Lambda^{-1}}.\nonumber
\end{align*}
Therefore, 
\begin{align*}
 &\|\vx^{k+1}-\vx^*\|^2_{\Lambda^{-1}}+\|\vd^{k+1}-\vd^*\|_{\vM+\Lambda}^2 \\
\leq		& (1-(2-\max_i(\alpha_iL_i))\min_i(\mu_i\alpha_i))\|\vx^{k}-\vx^*\|^2_{\Lambda^{-1}}\\
& +{\lambda_{\max} (\Lambda^{-1/2}\vM\Lambda^{-1/2})\over \lambda_{\max}(\Lambda^{-1/2}\vM\Lambda^{-1/2})+1}\|\vd^{k}-\vd^*\|_{\vM+\Lambda}^2.  
\end{align*}
Since 
\begin{align*}
&{\lambda_{\max} (\Lambda^{-1/2}\vM\Lambda^{-1/2})\over \lambda_{\max}(\Lambda^{-1/2}\vM\Lambda^{-1/2})+1} \\
= &1-{c\over \lambda_{\max} (\Lambda^{-1/2}(\vI-\vW)^\dagger\Lambda^{-1/2})}.
\end{align*}
Let $\rho$ defined as~\eqref{def:rho}, then we have~\eqref{for:linear_conv}.
\end{IEEEproof}

\begin{remark}\label{remark4}
The condition $\vI\succcurlyeq c\Lambda^{1/2}(\vI-\vW)\Lambda^{1/2}$ implies that $c\leq \lambda_{n-1}(\Lambda^{-1/2}(\vI-\vW)^\dagger\Lambda^{-1/2})$. 

\begin{itemize}
\item If the agents across the whole network use an identical stepsize $\alpha$, that is, ${\Lambda}=\alpha\vI$, then 
\begin{align*}
\begin{split}
\rho =\max & \left(1-(2-\alpha\max_iL_i)\alpha\min_i\mu_i,\right. \\ & \left. 1-{c\alpha\over \lambda_{\max}((\vI-\vW)^\dagger)} \right).
\end{split}
\end{align*}
A concise but informative expression of the rate $\rho =\max\left(1-{\min_i\mu_i\over \max_iL_i},{\lambda_2(\vW)-\lambda_n(\vW)\over 1-\lambda_n(\vW)} \right)$ can be obtained when we specifically choose $\alpha={1\over \max_i L_i}$ and $c={1\over ( 1-\lambda_n(\vW))\alpha}$. 
When $\lambda_{n}(\vW)$ is not given,  we choose $c={1/(2\alpha)}$ and obtain the scalability $\max\left({\max_iL_i\over\min_i\mu_i },{2\over1-\lambda_2(\vW)} \right)$. 
In this case, the network impact and the functional impact are decoupled.

\item If we let ${\Lambda}=\vL^{-1}$ and $c=\lambda_{n-1} (\vL^{-1/2}(\vI-\vW)^\dagger\vL^{-1/2})$, then the rate becomes 
\begin{align*}
\begin{split}
\rho =\max & \left(1-\min_i{\mu_i\over L_i},\right. \\ 
&\left.1-{\lambda_{n-1} (\vL^{1/2}(\vI-\vW)^\dagger\vL^{1/2})\over \lambda_{\max}(\vL^{1/2}(\vI-\vW)^\dagger\vL^{1/2})}  \right).
\end{split}
\end{align*}
When $\lambda_{n}(\vW)$ is not given, we choose $c=1/(2\max_i\alpha_i)=\min_i L_i/2$ and obtain the scalability $\max\left(\max_i\frac{L_i}{\mu_i},\frac{\max_i L_i}{\min_i L_i}\cdot\frac{2}{1-\lambda_2(\vW)}\right)$. 
In this case, the networking impact is coupled with the function factors, i.e., the smoothness heterogeneity $\frac{\max_i L_i}{\min_i L_i}$ is multiplied on the networking impact. 
While the other number depends on the functional condition numbers $\frac{L_i}{\mu_i}$'s only.
\end{itemize}
\end{remark}

\begin{remark}
Theorem~\ref{thm:linear_conv} separates the dependence of the linear convergence rate on the functions and the network structure. 
In our current scheme, all the agents perform information exchange and the proximal-gradient step once in each iteration. If the proximal-gradient step is expensive, this explicit rate formula can help us to decide whether the so-called multi-step consensus can help reducing the computational time. 

For the sake of simplicity, let us assume for this moment that all the agents have the same strong convexity constant $\mu$ and gradient Lipschitz continuity constant $L$. Suppose that the ``$t$-step consensus'' technique is employed, i.e., the mixing matrix $\vW$ in our algorithm is replaced by $\vW^t$, where $t$ is a positive integer. Then to reach $\epsilon$-accuracy, the number of iterations needed is
\[O\left(\max\left(\frac{L}{\mu},\frac{1-\lambda_n(\vW^t)}{1-\lambda_2(\vW^t)}\right)\right)\log\frac{1}{\epsilon}.\]
When $L/\mu=1$ and step sizes are chosen as $\Lambda=\vL^{-1}$, it says that we should let $t\rightarrow+\infty$ if the graph is not a complete graph. Such theoretical result is correct in intuition since in this case, the centralized gradient descent only needs one step to reach optimal and the bottleneck in decentralized optimization is the network.

Suppose $t_{\max}$ is a reasonable upper bound on $t$, which is set by the system designer. It is difficult to explicitly find an optimal $t$. But with the above analysis as an evidence, we suggest that one choose $t=\min\left([\log_{\lambda_2(\vW)}(1-\frac{\mu}{L})],t_{\max}\right)$ if $1-\frac{\mu}{L}>\lambda_2(\vW)$; otherwise $t=1$. Here $[\cdot]$ gives the nearest integer. 

If the bottleneck is on the functions, we can introduce a mapping $x=By$ and change the unknown variable from $x$ to $y$. 
E.g., if the function $s_i(x)$ is a composition of a convex function with a linear mapping, replacing $x$ using $y$ changes the linear mapping and the condition number $L/\mu$ of the function.  
When $B$ is diagonal, it is similar to the column normalization in machine learning applications. 
There are other possible ways for reducing the condition number of the functions. 
It is out of the scope of this work, and we leave this as future work. 

\end{remark}
%
%

%
%
%

\section{Numerical Experiments}\label{sec:num_exp}
In this section, we compare the performance of NIDS with several state-of-the-art algorithms for decentralized optimization. These methods are
\begin{itemize}
	\item The EXTRA/PG-EXTRA (see \eqref{pg-extra});
	\item The DIGing-ATC~\cite{nedic2016geometrically}. For reference, the DIGing-ATC updates are provided as follows:
		\begin{align*}
		\vx^{k+1} =& \vW(\vx^k-\alpha\vy^k),\\
		\vy^{k+1} =& \vW(\vy^k + \nabla f(\vx^{k+1}) - \nabla f(\vx^k)).
		\end{align*}
\item The accelerated distributed Nesterov gradient descent (Acc-DNGD-SC in~\cite{qu2017accelerated});
\item The (dual friendly) optimal algorithm (OA) for distributed optimization (equation (7) in~\cite{uribe2017optimal}).
\end{itemize}
Note there are two rounds of communication in each iteration of DIGing-ATC and Acc-DNGD-SC while there is only one round in that of EXTRA/NIDS/OA. 
For all the experiments, we first compute the exact solution $\vx^*$ for~\eqref{eq:F} using the centralized (proximal) gradient descent. 
All networks are randomly generated with connectivity ratio $\tau$, where
$\tau$ is defined as the number of actual edges divided by the total number of possible edges ${{n(n-1)}\over 2}$. 
We will report the specific $\tau$ used in each test. The mixing matrix $\vW$ is always chosen with the Metropolis rule (see~\cite{boyd2004fastest} and~\cite[Section 2.4]{Shi2015}).

The experiments are carried in Matlab R2016b running on a laptop with Intel i7 CPU @ 2.60HZ, 16.0 GB of RAM, and Windows 10 operating system. 
The source codes for reproducing the numerical results can be accessed at~\href{https://github.com/mingyan08/NIDS}{https://github.com/mingyan08/NIDS}.

\subsection{The strongly convex case with $r(\vx)=0$}\label{sec:num1} 
Consider the decentralized problem that solves for an unknown signal $x \in \RR^{p}$. Each agent $i \in \{1,\cdots,n\}$ takes its own measurement via $y_i=\vM_ix+e_i$, where $y_i \in \RR^{m_i}$ is the measurement vector, $\vM_i \in \RR^{m_i \times p} $ is the sensing matrix, and $e_i \in \RR^{m_i} $ is the independent and identically distributed noise. 
To estimate $x$ collaboratively, we apply the decentralized algorithms to solve
\begin{align*}
\Min_x~&\frac{1}{n}\sum_{i=1}^{n}{1\over2}\|\vM_ix-y_i\|^2 
\end{align*}

In order to ensure that each function ${1\over 2}\|\vM_ix-y_i\|^2$ is strongly convex, we choose $m_i=60$ and $p=50$ and set the number of nodes $n=40$.
For the first experiment, we choose $\vM_i$ such that the Lipshchitz constant of $\nabla s_i$ satisfies $L_i=1$ and the strongly convex constant $\mu_i=0.5$ for all $i$. 
Based on Remark~\ref{remark4}, we choose $\alpha={1/(\max_i L_i)}=1$ and $c={1/(1-\lambda_n(\vW))}$ for NIDS. 
In addition, we choose $c={1/2}$ such that $\widetilde{\vW}={\vI+\vW\over 2}$, which gives the same as that for EXTRA. 

The comparison of these methods (NIDS with $c={1/((1-\lambda_n(\vW))\alpha)}$, NIDS with $c=1/2$, EXTRA, DIGing-ATC, Acc-DNGD-SC, and OA) is shown in Fig.~\ref{fig:compare_c_noR} for two different networks with connectivity ratios $\tau=0.35$ (top) and $\tau=0.45$ (bottom), respectively. 
It shows better performance of NIDS in both choices of $c$ (corresponding to known $\vW$ and unknown $\vW$) than that of other algorithms. 
NIDS with $c={1/((1-\lambda_n(\vW))\alpha)}$ always takes less than half the number of iterations used by EXTRA to reach the same accuracy. In our experiment, DIGing-ATC appears to be sensitive to networks. 
Under a better connected network (see Fig.~\ref{fig:compare_c_noR} bottom), DIGing-ATC can catch up with NIDS with $c={1/(2\alpha)}$. 
The theoretical step-size of Acc-DNGD-SC is too small due to a very small constant in the bound in~\cite{qu2017accelerated}, and the convergence of Acc-DNGD-DC under such theoretical step-size in our test is slow and uncompetitive. 
Thus we have carefully tuned its step-size. 
With the hand-optimized step-size, Acc-DNGD-SC can achieve a comparable performance as NIDS with $c={1/(2\alpha)}$. 
In the plots, we observe that OA is fast in terms of the number of iterations. However, in this case, the per-iteration cost of OA is relatively high since it requires solving a system of linear equations at each iteration (though factorization tricks may be used to save some computational time). 
\begin{figure}[!ht]
	\begin{center}
		\includegraphics[width=0.4\textwidth]{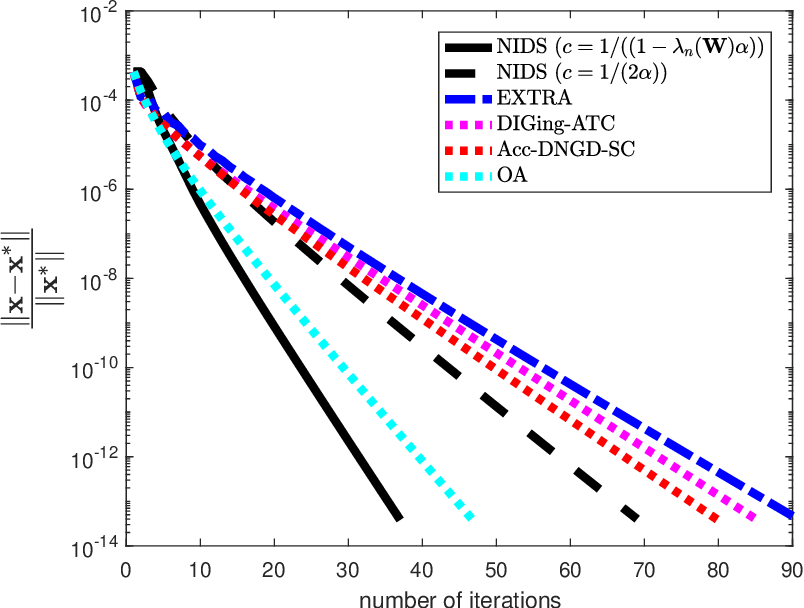}
\includegraphics[width=0.4\textwidth]{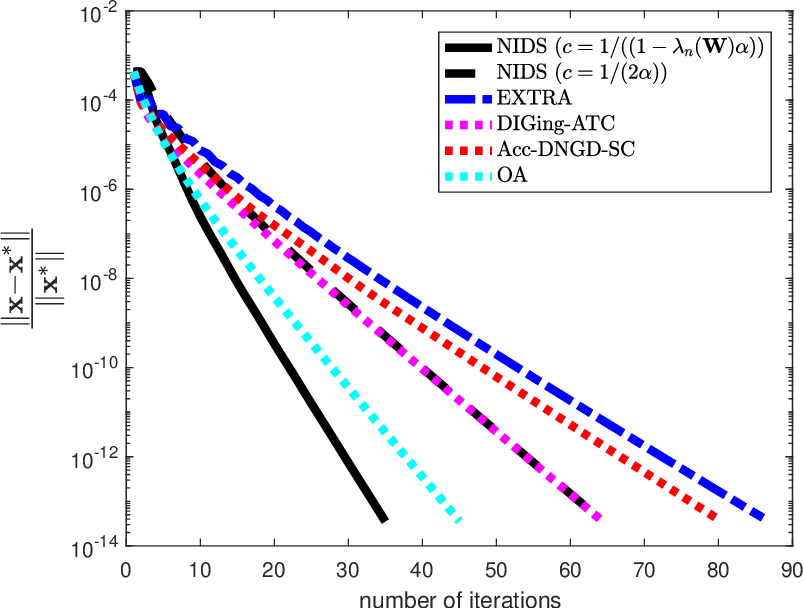}
	\end{center}
	\caption{\label{fig:compare_c_noR} Relative error $\frac{\left\| \vx-\vx^{*}\right\|}{\left\| \vx^{*}\right\|}$ against the number of iterations for two different networks (top: $\tau=0.35$; bottom: $\tau=0.45$). NIDS, EXTRA, and DIGing-ATC use the same step-size $\alpha=1 /{L}$, where $L =\max_iL_i$. 
    The step-size for Acc-DNGD-SC is hand-optimized. 
We use the default step-sizes for OA as suggested by the authors.}
\end{figure}


Next, we demonstrate the effort of uncoordinated/adaptive step-size. 
We construct the function with $\mu_i=0.02$ and $L_i=1$ for each node $i$. 
Then we change the $L_i$ values by multiplying the function by a constant.
%
\begin{figure}[!ht]
	\begin{center}
		\includegraphics[width=0.4\textwidth]{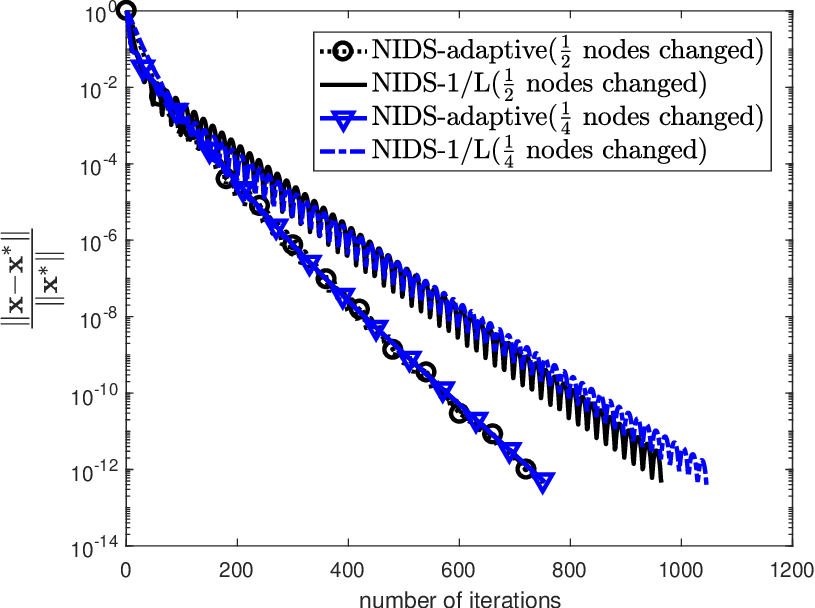}
	\end{center}
	\caption{\label{fig:compare_adpt_noR} The relative error $\frac{\left\| \vx-\vx^{*}\right\|}{\left\| \vx^{*}\right\|}$ against the number of iterations. NIDS-1/L uses the same step-size $1 /{L}$, where $L =\max_iL_i$, and NIDS-adaptive uses the step-size $1 / {L_i}$ for each node.  We assume that no graph information is available, thus $c=1/(2 \max_i \alpha_i)$.
     The connectivity ratio of the network is set as $\tau=0.1$.}
\end{figure}
We use the same mixing matrix for the following two experiments. 
\begin{itemize}
	\item We change half nodes. 
	We randomly pick an even number node and multiply its function by $4$. 
	For remaining even number nodes, we multiply their functions by a random integer 2 or 3.
	\item We change a quarter nodes.  
	We randomly pick an node not in $\cU=\{4,8,16,\dots,40\}$ and multiply its function by 10. 
	Then for other nodes in $\cU$, we multiply their functions by a random integer between 2 and 9.
\end{itemize} 

We compare NIDS with adaptive step-size ($1 / {L_i}$ for node $i$) and NIDS with same step-size $1 / {\max_iL_i}$ in Fig.~\ref{fig:compare_adpt_noR}. 
We let $c=1/(2 \max_i \alpha_i)$, so no network information is needed.
As shown in Fig.~\ref{fig:compare_adpt_noR}, NIDS with adaptive step-size converges faster than same step-size. 

\subsection{The case with nonsmooth function $r(\vx)$}
In this subsection, we compare the performance of NIDS with PG-EXTRA~\cite{shi2015proximal} only since other methods in Section~\ref{sec:num1}, such as DIGing, can not be applied to this nonsmooth case.
We consider a decentralized compressed sensing problem. 
Again, each agent $i \in \{1,\cdots,n\}$ takes its own measurement via $y_i=\vM_ix+e_i$, where $y_i \in \RR^{m_i}$ is the measurement vector, $\vM_i \in \RR^{m_i \times p} $ is the sensing matrix, and $e_i \in \RR^{m_i} $ is the independent and identically distributed noise. 
Here, $x$ is a sparse signal. 
The optimization problem is
\begin{align*}
\Min_x~
&\frac{1}{n}\sum_{i=1}^{n}{1\over2}\|\vM_ix-y_i\|^2+\frac{1}{n}\sum_{i=1}^{n}\lambda_i\|x\|_1, 
\end{align*}
where the connectivity ratio of the network $\tau=0.1$. We normalize the problem to make sure that the Lipschitz constant satisfies $L_i=1$ for each node, we choose $m_i=3$ and $p=200$ and set the number of nodes $n=40$. 

Fig.~\ref{fig:mulStep_R} shows that a larger step-size in NIDS leads to faster convergence. With step-size $1$, NIDS and PG-EXTRA converge at the same speed. But if we keep increasing the step-size, PG-EXTRA will diverge with step-size $1.4$ while the step-size of NIDS can be increased to $1.9$ maintaining convergence at a faster speed.

\begin{figure}[!ht]
	\begin{center}
	\vspace{-1em}
	\includegraphics[width=0.4\textwidth]{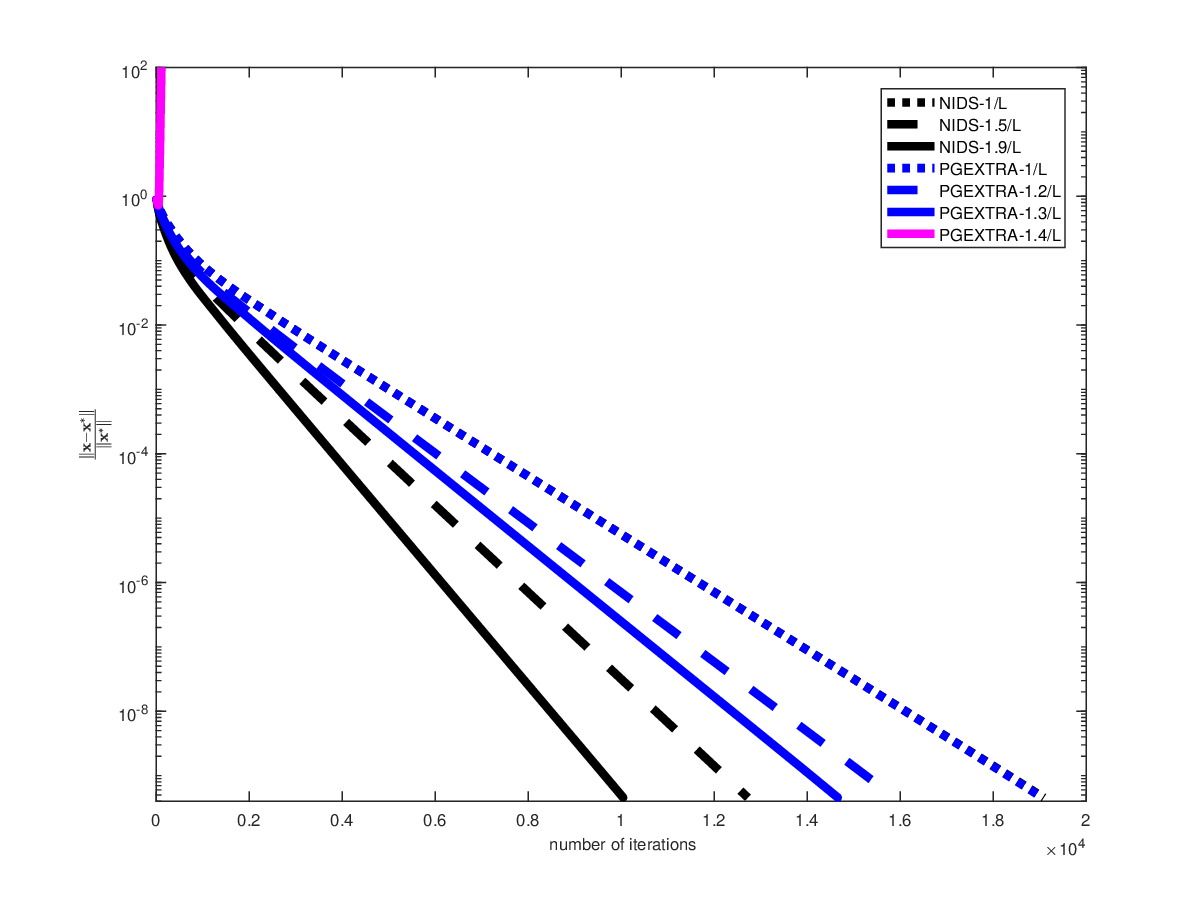}
	\end{center}
	\caption{\label{fig:mulStep_R}  The relative error $\frac{\left\| \vx-\vx^{*}\right\|}{\left\| \vx^{*}\right\|}$ against the number of iterations. Different step-sizes for PG-EXTRA and NIDS are considered. 
	For instance, ``NIDS-$1/{L}$'' is NIDS using the same step-size $1 /{L}$ across the network of agents, where $L =\max_iL_i$. The connectivity ratio of the network is $\tau=0.4$.}
\end{figure}

\subsection{An application in classification for healthcare data}

\begin{figure*}[!ht]
	\begin{center}
		\includegraphics[width=0.28\textwidth]{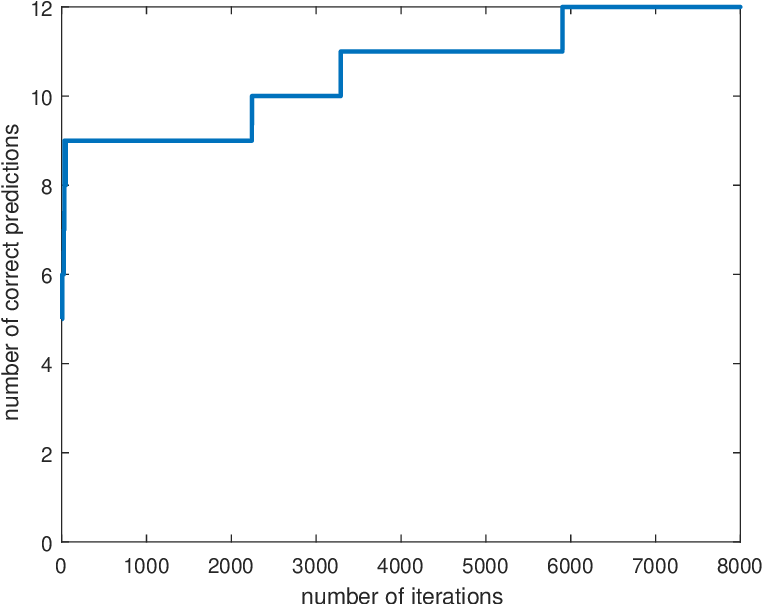}
 		\includegraphics[width=0.28\textwidth]{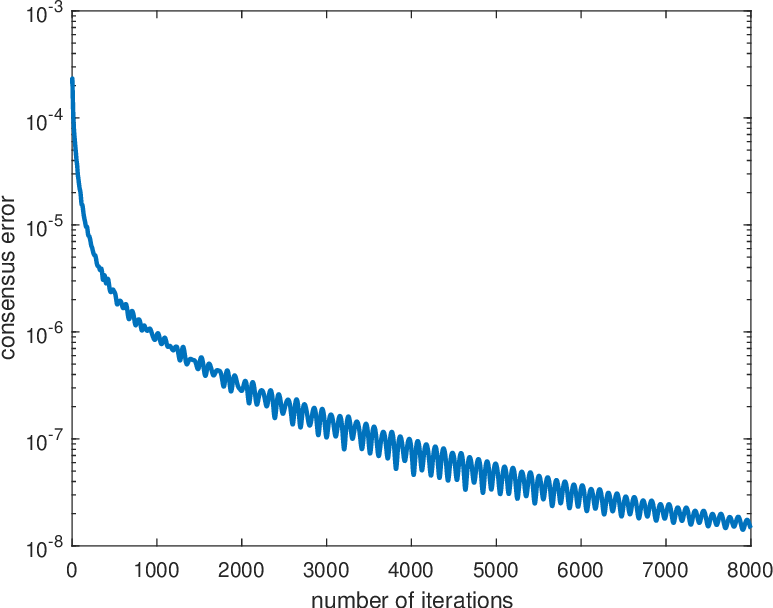}
		\includegraphics[width=0.28\textwidth]{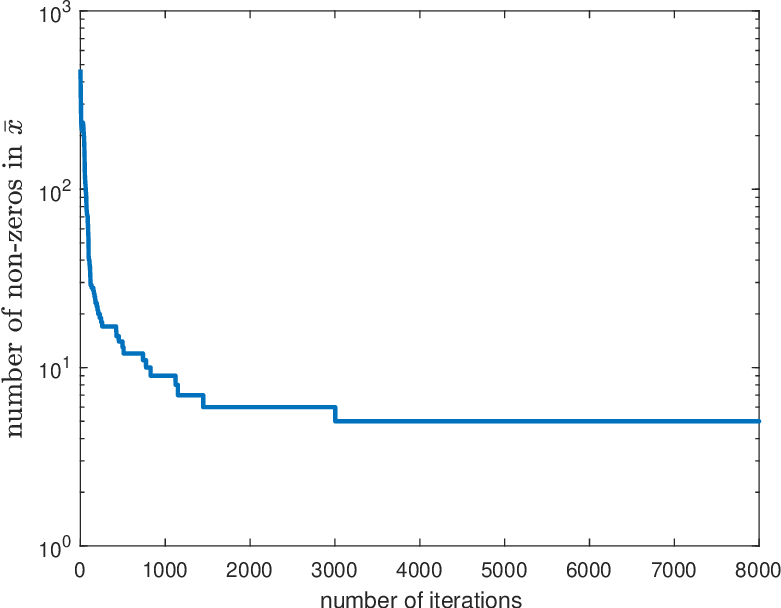}
	\end{center}
	\caption{\label{fig:logistic_reg} {Performance of NIDS for sparse logistic regression. Left: number of correct predictions vs. iteration. 
	Middle: consensus error $\left\|\vx^k\right\|_{I-\vW}$ vs. iteration; Right: number of non-zero elements in $\bar{x}^k={1\over50}\sum_{i=1}^{50}x_i^k$ vs. iteration.}}
\end{figure*}

We consider a decentralized sparse logistic regression problem to classify the colon-cancer data~\cite{liu2009large}. 
There are 62 samples, and each sample features 2,000 pieces of gene expressing information (numericalized and normalized~\cite{liu2009large}) and a binary outcome. 
The outcome can be normal/negative ($+1$) or tumor/positive ($-1$) and the data set contains $22$ normal and $40$ tumor colon tissue samples. 
We store the gene information in $\vM_i\in\RR^{1\times 2001}$ (one more dimension is augmented to take care of the linear offset/constant in the logit function model), and the outcome information is $y_i\in\{-1,1\}$, $i\in\mathcal{S}_1\bigcup\mathcal{S}_2$, where $\mathcal{S}_1$ serves for training while $\mathcal{S}_2$ serves for testing. 
Suppose we have a 50-node connected network where each node $i$ holds $1$ sample $(\vM_i,y_i)$ (the connected network is randomly generated and its connectivity ratio is set to $0.08$; the 50 in-network samples indexed by $\mathcal{S}_1$ are randomly drawn from the 62 samples). 
The decentralized logistic regression

\begin{align*}
\Min_x~
&\frac{1}{|\mathcal{S}_1|}\sum_{i\in\mathcal{S}_1}\ln(1+\exp(-\vM_i x_i y_i))\\
&+\frac{1}{|\mathcal{S}_1|}\sum_{i\in\mathcal{S}_1}\hat{\lambda}_i\|x_i\|_2^2+\frac{1}{|\mathcal{S}_1|}\sum_{i\in\mathcal{S}_1}\lambda_i\|x_i\|_1,
\end{align*}
is then solved over the network to train a sparse linear classifier $\vx^*$ for the outcome prediction of the remaining/future samples/data. 
In the optimization formulation, the $\ell_2$ norm term imposes strong convexity to $s$, while $\ell_1$ term promotes sparsity of the solution. 
Aside from the 50 samples used for training purpose, we randomly select 12 nodes from the 50 nodes to show the prediction performance of the remaining 12 samples in Fig.~\ref{fig:logistic_reg} left. 
The middle and right figures in Fig.~\ref{fig:logistic_reg} show how the consensus error $\left\|\vx^k\right\|_{I-\vW}$ and the sparsity of the average solution vector ${1\over 50}\sum_{i=1}^{50} x_i^k$ drops, respectively.


\section{Conclusion}\label{sec:concl}
We proposed a novel decentralized consensus algorithm NIDS, whose step-size does not depend on the network structure. 
In NIDS, the step-size depends {\it only} on the objective function, and it can be as large as $2/L$, where $L$ is the Lipschitz constant of the gradient of the smooth function. 
We showed that NIDS converges at the $o(1/k)$ rate for the general convex case and at a linear rate for the strongly convex case. 
For the strongly convex case, we separated the linear convergence rate's dependence on the objective function and the network. 
The separated convergence rates match the typical rates for the general gradient descent and the consensus averaging. 
Furthermore, every agent in the network can choose its own step-size independently by its own objective function. 
Numerical experiments validated the theoretical results and demonstrated better performance of NIDS over state-of-the-art algorithms.
Because the step-size of NIDS does not depend on the network structure, there are many possible future extensions. 
One extension is to apply NIDS on dynamic networks where nodes can join and drop off.

\section*{Acknowledgement}
We thank the anonymous reviewers for helpful comments and suggestions to improve the clarity of this paper.


\bibliographystyle{IEEEtran}
\bibliography{nids_bib}


\begin{IEEEbiography}
[{\includegraphics[width=1in,height=1.25in,clip,keepaspectratio]{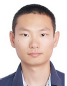}}]
{Zhi Li} received the B.S. and the M.S. degree in Applied Mathematics from China University of Petroleum, Shandong, China, in 2007 and 2010, respectively. Then, he participated in the cooperative education program and
also received the M.S. degree in Applied Science from Saint Mary’s University, Halifax, NS, Canada, in 2012. After being awarded the Hong Kong Ph.D. Fellowship, he went to Hong Kong Baptist University, HK,
China, where he received his Ph.D. in Applied Mathematics in 2016, and he was supported by the fellowship from 2013 to 2016. Since 2016 he has been a postdoctoral
researcher with the Department of Computational Mathematics, Science and Engineering, Michigan State University, East Lansing, MI.
\end{IEEEbiography}

\begin{IEEEbiography}
	[{\includegraphics[width=1in,height=1.25in,clip,keepaspectratio]{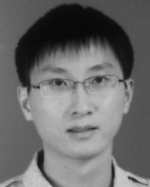}}]
    {Wei Shi} received the B.E. degree in automation and the Ph.D. degree in control science and engineering from the University of Science and Technology of China, Hefei, in 2010 and 2015, respectively. 
    He was a Postdoctoral Researcher with the Coordinated Science Laboratory, the University of Illinois at Urbana-Champaign, Urbana, IL, USA, from 2015 to 2016, with Arizona State University, Tempe, AZ, USA from 2016 to 2018, and with Princeton University from 2018-2019. His research interests spanned in optimization, cyber–physical systems, and big data analytics. He was awarded the 2017 Young Author Best Paper Award from the IEEE Signal Processing Society.
\end{IEEEbiography}
\begin{IEEEbiography}
	[{\includegraphics[width=1in,height=1.25in,clip,keepaspectratio]{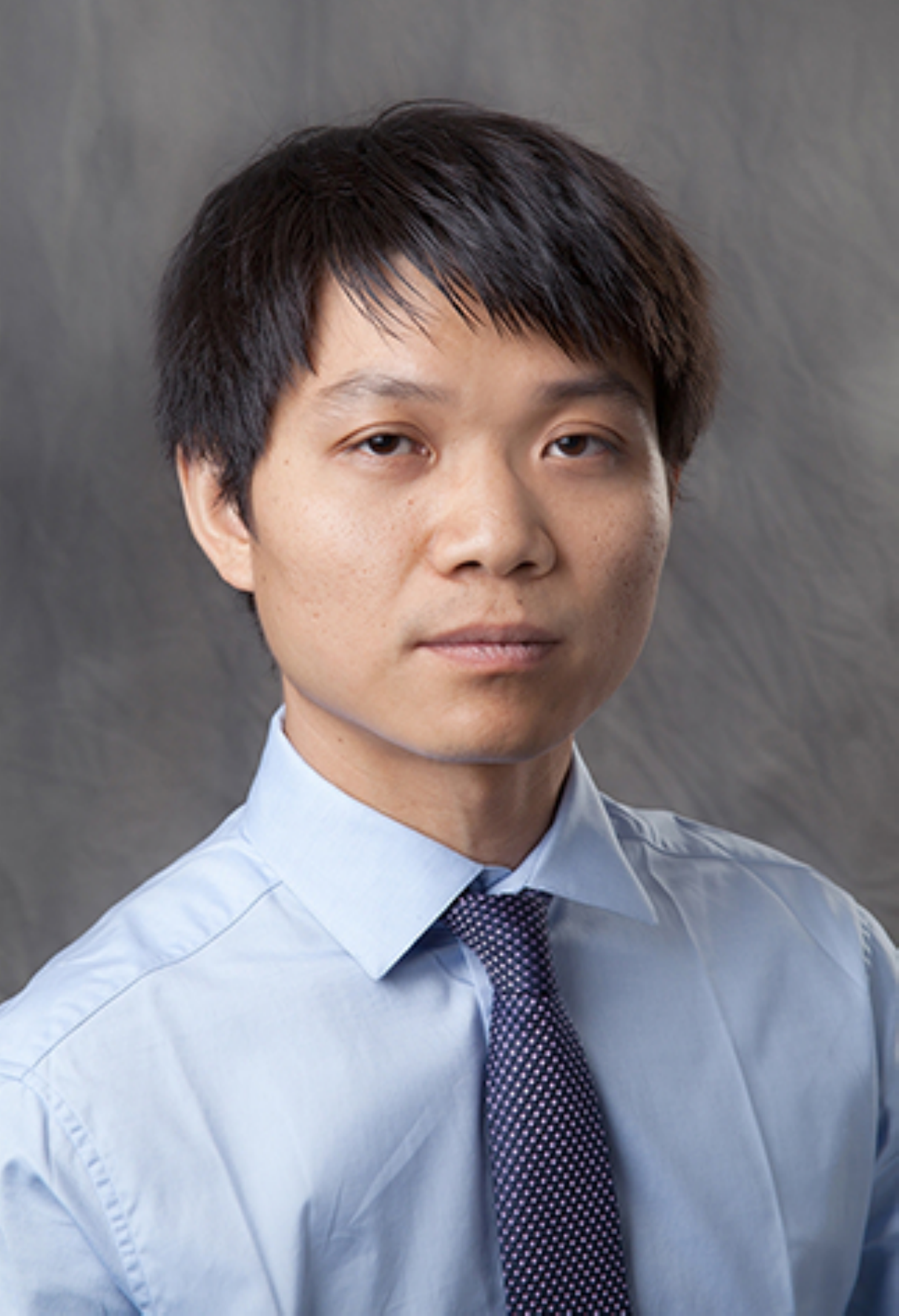}}]
	{Ming Yan} is currently an Assistant Professor in the Department of Computational Mathematics, Science and Engineering and the Department of Mathematics, Michigan State University. He received the B.S. Degree and M.S. Degree from University of Science and Technology of China, and Ph.D. degree from University of California, Los Angeles (UCLA) in 2012. His research interests include optimization methods and their applications in sparse recovery and regularized inverse problems, variational methods for image processing, parallel and distributed algorithms for solving big data problems.
\end{IEEEbiography}





\newpage

\markboth{L. Zhi, W. Shi, and M. Yan}{}
\setcounter{page}{1}
\section*{Supplementary Material for ``A Decentralized Proximal-Gradient Methodwith Network Independent Step-sizes andSeparated Convergence Rates''}
\subsection{Proof of Lemma~\ref{lemma:fixedpoint}}
\setcounter{lemma}{0}

\begin{lemma}[Fixed point of~\eqref{for:alg2}]
$(\vd^*,\vz^*)$ is a fixed point of~\eqref{for:alg2} if and only if there exists a subgradient $\vq^*\in\partial r(\vx^*)$ such that $\vz^*=\vx^*+\Lambda\vq^*$ and
\begin{subequations}
\begin{align*}
&\vd^* + \nabla s(\vx^*)+\vq^* & = \vzero ,\\
&(\vI-\vW)\vx^* &=\vzero.
\end{align*}
\end{subequations}
\end{lemma}

\begin{IEEEproof}
``$\Rightarrow$" 
If $(\vd^*,\vz^*)$ is a fixed point of~\eqref{for:alg2}, we have 
\begin{align*}
\vzero &=  c(\vI-\vW)\left(2\vx^{*}-\vz^{*}-\Lambda\nabla s\left(\vx^{*}\right)-\Lambda\vd^*\right)\\
&=  c(\vI-\vW)\vx^{*},
\end{align*}
where the two equalities come from~\eqref{for:alg2_d} and~\eqref{for:alg2_z}, respectively.
Combining~\eqref{for:alg2_z} and~\eqref{for:alg2_x} gives
\begin{align*}
\vzero &= \vz^* - \vx^{*}+\Lambda\nabla s(\vx^{*})+\Lambda\vd^{*} \\
&= \Lambda (\vq^*+\nabla s(\vx^{*})+\vd^*),
\end{align*}
where $\vq^*\in\partial r(\vx^*)$.

``$\Leftarrow$"
In order to show that $(\vd^*,\vz^*)$ is a fixed point of iteration~\eqref{for:alg2}, we just need to verify that $(\vd^{k+1},\vz^{k+1})=(\vd^*,\vz^*)$ if $(\vd^k,\vz^k)=(\vd^*,\vz^*)$. 
From~\eqref{for:alg2_x}, we have $\vx^k=\vx^*$, then
\begin{align*}
\vd^{k+1} &=  \vd^* + c(\vI-\vW)\left(2\vx^*-\vz^*-\Lambda\nabla s\left(\vx^*\right)-\Lambda\vd^*\right)\\
&=  \vd^* + c(\vI-\vW)\vx^*=\vd^*,\\
\vz^{k+1} &=  \vx^*-\Lambda\nabla s\left(\vx^*\right)-\Lambda\vd^{*} =\vx^*+\Lambda\vq^*=\vz^*. 
\end{align*}
Therefore, $(\vd^*,\vz^*)$ is a fixed point of iteration~\eqref{for:alg2}.
\end{IEEEproof} 

\subsection{Proof of Lemma~\ref{lemma:condition}}
\begin{lemma}[Optimality condition]
$\vx^*$ is consensual with $x^*_1=x^*_2=\cdots=x^*_n=x^*$ being an optimal solution of problem~\eqref{eq:F} if and only if there exists $\vp^*$ and a subgradient $\vq^*\in\partial r(\vx^*)$  such that:
\begin{subequations}
\begin{align}
&(\vI-\vW)\vp^* + \nabla s(\vx^*)+\vq^* &= \vzero ,\label{for:opt1_sm}\\
&(\vI-\vW)\vx^*&=\vzero .\label{for:opt2_sm}
\end{align}
\end{subequations}
In addition, $(\vd^*=(\vI-\vW)\vp^*, \vz^*=\vx^*+\Lambda\vq^*)$ is a fixed point of iteration~\eqref{for:alg2}.
\end{lemma}
\begin{IEEEproof}
``$\Rightarrow$" 
Because $\vx^*=\vone_{n\times 1} (x^*)^\top$,  we have $(\vI-\vW)\vx^*=(\vI-\vW)\vone_{n\times 1} (x^*)^\top=\vzero_{n\times 1}(x^*)^\top=\vzero$. 
The fact that $x^*$ is an optimal solution of problem~\eqref{eq:F} means there exists $\vq^*\in\partial r(\vx^*)$ such that $(\nabla s(\vx^*)+\vq^*)^\top\vone_{n\times 1} = \vzero$. 
That is to say all columns of $\nabla s(\vx^*)+\vq^*$ are orthogonal to $\vone_{n\times 1}$.
Therefore, Remark~\ref{remark:W} shows the existence of $\vp^*$ such that $(\vI-\vW)\vp^*+\nabla s(\vx^*)+\vq^*=\vzero$.

``$\Leftarrow$"
Equation~\eqref{for:opt2_sm} shows that $\vx^*$ is consensual because of item 3 of Assumption~\ref{assu:1}, i.e., $\vx^*=\vone_{n\times 1}(x^*)^\top$ for some $x^*$.
From~\eqref{for:opt1_sm}, we have $\vzero =((\vI-\vW)\vp^* + \nabla s(\vx^*)+\vq^*)^\top \vone_{n\times 1}=(\vp^*)^\top(\vI-\vW) \vone_{n\times 1}+(\nabla s(\vx^*)+\vq^*)^\top \vone_{n\times 1}=(\nabla s(\vx^*)+\vq^*)^\top \vone_{n\times 1}$. 
Thus, $\vzero\in \sum_{i=1}^n(\nabla s_i(x^*)+\partial r_i(x^*))$ because $\vx^*$ is consensual. This completes the proof for the equivalence.

Lemma~\ref{lemma:fixedpoint} shows that $(\vd^*=(\vI-\vW)\vp^*, \vz^*=\vx^*+\Lambda\vq^*)$ is a fixed point of iteration~\eqref{for:alg2}.
\end{IEEEproof}

\subsection{Proof of Lemma~\ref{lemma:equiv_norm}}
\begin{lemma}[Norm over range space]
For any symmetric positive semidefinite matrix $\vA\in \RR^{n\times n}$ with rank $r\leq n$, 
let $\lambda_1\geq \lambda_2\geq \dots\geq \lambda_r>0$ be its $r$ eigenvalues. 
Then $\range(\vA)$ defined in Section~\ref{sec:notation} is a $rp$-dimensional subspace in $\RR^{n\times p}$ and has a norm defined by $\|\vx\|^2_{\vA^\dagger}:=\langle \vx, \vA^\dagger \vx\rangle$, where $\vA^\dagger$ is the pseudo inverse of $\vA$. 
In addition, $\lambda_1^{-1}\|\vx\|^2\leq \|\vx\|^2_{\vA^\dagger}\leq \lambda_r^{-1}\|\vx\|^2$ for all $\vx\in\range(\vA)$.
\end{lemma}
\begin{IEEEproof} 
Let $\vA=\vU\Sigma \vU^\top$, where $\Sigma=\Diag(\lambda_1,\lambda_2,\cdots,\lambda_r)$ and the columns of $\vU$ are orthonormal eigenvectors for corresponding eigenvalues, i.e., $\vU\in \RR^{n\times r}$ and $\vU^\top \vU=\vI_{r\times r}$.
Then $\vA^\dagger=\vU\Sigma^{-1}\vU^\top$, where $\Sigma^{-1} = \Diag(\lambda_1^{-1},\lambda_2^{-1},\cdots,\lambda_r^{-1})$. 

Letting $\vx=\vA \vy$, we have $\|\vx\|^2 = \langle \vU\Sigma \vU^\top \vy, \vU\Sigma \vU^\top \vy\rangle = \langle \Sigma \vU^\top \vy, \Sigma \vU^\top \vy\rangle = \|\Sigma \vU^\top \vy\|^2$. In addition, 
\begin{align*}
\langle \vx, \vA^\dagger \vx\rangle=&\langle \vA \vy, \vA^\dagger \vA \vy\rangle \\= & \langle \vU\Sigma \vU^\top \vy, \vU\Sigma^{-1}\vU^\top \vU\Sigma \vU^\top \vy\rangle \\= &\langle \Sigma \vU^\top \vy, \Sigma^{-1}\Sigma \vU^\top \vy\rangle.
\end{align*}
Therefore, 
\begin{align}\label{for:equiv_norm}
\lambda^{-1}_1 \|\vx\|^2 &=\lambda^{-1}_1 \|\Sigma \vU^\top \vy\|^2 \\& \leq \langle \vx, \vA^\dagger \vx\rangle \leq \lambda^{-1}_r \|\Sigma \vU^\top \vy\|^2 = \lambda^{-1}_r \|\vx\|^2, \nonumber
\end{align}
which means that $\|\cdot\|^2_{\vA^\dagger}=\langle \cdot, \vA^\dagger\cdot\rangle$ is a norm for $\range(\vA)$.
\end{IEEEproof}

\subsection{Proof of Proposition~\ref{prop:M} }
\setcounter{proposition}{0}
\begin{proposition} 
Let $\vM=c^{-1}(\vI-\vW)^\dagger-\Lambda$ with $\Lambda$ being symmetric positive definite and $\vI\succcurlyeq c\Lambda^{1/2}(\vI-\vW)\Lambda^{1/2} \succcurlyeq\vzero$. Then $\|\cdot\|_{\vM}$ is a norm defined for $\range(\vI-\vW)$.
\end{proposition}
\begin{IEEEproof}
Rewrite the matrix $\vM$ as 
\begin{align*}
\vM = & c^{-1}(\vI-\vW)^\dagger-\Lambda \\ 
= &\Lambda^{1/2}(c^{-1}\Lambda^{-1/2}(\vI-\vW)^\dagger\Lambda^{-1/2}-\vI)\Lambda^{1/2}\\
    = & \Lambda^{1/2}((c\Lambda^{1/2}(\vI-\vW)\Lambda^{1/2})^\dagger-\vI)\Lambda^{1/2}.
\end{align*}
For any $\vx\in\range(\vI-\vW)$, we can find $\vy\in\RR^{n\times p}$ such that $\vx=(\vI-\vW)\Lambda^{1/2}\vy$. 
Then 
\begin{align*}
&\langle \vx,\vM\vx\rangle \\= & \langle (\vI-\vW)\Lambda^{1/2}\vy,\Lambda^{1/2}((c\Lambda^{1/2}(\vI-\vW)\Lambda^{1/2})^\dagger\\
&-\vI)\Lambda^{1/2} (\vI-\vW)\Lambda^{1/2}\vy\rangle \\
 = & \langle \Lambda^{1/2}(\vI-\vW)\Lambda^{1/2}\vy,((c\Lambda^{1/2}(\vI-\vW)\Lambda^{1/2})^\dagger\\
&-\vI)\Lambda^{1/2} (\vI-\vW)\Lambda^{1/2}\vy\rangle.
\end{align*}
We apply Lemma~\ref{lemma:equiv_norm} on $\Lambda^{1/2}(\vI-\vW)\Lambda^{1/2}$ and obtain the result.
\end{IEEEproof}

\subsection{Proof of Theorem~\ref{lemma:conv_rate}}
Before proving Theorem~\ref{lemma:conv_rate}, we present two lemmas. 
The first lemma shows that the distance to a fixed point of~\eqref{for:alg2} is decreasing, and the second one show the distance between two iterations is decreasing.

\begin{lemma}[A key inequality of descent]\label{lemma:conv}
Let $(\vd^*,\vz^*)$ be a fixed point of~\eqref{for:alg2} and $\vd^*\in\range(\vI-\vW)$. 
For the sequence $(\vd^k,\vz^k)$  generated from NIDS in~\eqref{for:alg2} with $\vI\succcurlyeq c\Lambda^{1/2}(\vI-\vW)\Lambda^{1/2}$, 
we have
\begin{align}\label{for:powerIneq2}
&\|\vz^{k+1}-\vz^*\|^2_{\Lambda^{-1}}+\|\vd^{k+1}-\vd^*\|_\vM^2 \nonumber\\
\leq		& \|\vz^{k}-\vz^*\|^2_{\Lambda^{-1}}+\|\vd^{k}-\vd^*\|_\vM^2 \\
& -(1-\max_i{\alpha_iL_i\over 2})\left(\|\vz^{k}-\vz^{k+1}\|^2_{\Lambda^{-1}}\right. \nonumber\\ &\left.+\|\vd^{k}-\vd^{k+1}\|_\vM^2\right).\nonumber
\end{align}
\end{lemma}

\begin{IEEEproof}
Young's inequality and~\eqref{for:cocoer} give us 
\begin{align*}
		 & 2\langle\nabla s\left(\vx^{k}\right)-\nabla s\left(\vx^{*}\right),\vz^{k}-\vz^{k+1}\rangle \\ &-2\langle \vx^{k}-\vx^{*},\nabla s(\vx^{k})-\nabla s(\vx^{*})\rangle \nonumber\\
\leq & {1\over 2}\|\vz^{k}-\vz^{k+1}\|^{2}_\vL +2\|\nabla s(\vx^{k})-\nabla s(\vx^{*})\|^{2}_{\vL^{-1}}\\ &-2\|\nabla s(\vx^{k})-\nabla s(\vx^{*})\|^{2}_{\vL^{-1}} \nonumber\\
=    & \frac{1}{2}\|\vz^{k}-\vz^{k+1}\|^{2}_\vL.
\end{align*}
Therefore, from~\eqref{for:powerIneq}, we have
\begin{align*}
        & \|\vz^{k+1}-\vz^*\|^2_{\Lambda^{-1}}+\|\vd^{k+1}-\vd^*\|_\vM^2 \nonumber\\
\leq		& \|\vz^{k}-\vz^*\|^2_{\Lambda^{-1}}+\|\vd^{k}-\vd^*\|_\vM^2 -\|\vz^{k}-\vz^{k+1}\|^2_{\Lambda^{-1}}\\ 
&-\|\vd^{k}-\vd^{k+1}\|_\vM^2  +\frac{1}{2}\|\vz^{k}-\vz^{k+1}\|^{2}_\vL\nonumber\\
\leq    & \|\vz^{k}-\vz^*\|^2_{\Lambda^{-1}}+\|\vd^{k}-\vd^*\|_\vM^2 \\ 
-&(1-\max_i{\alpha_iL_i\over 2})\|\vz^{k}-\vz^{k+1}\|^2_{\Lambda^{-1}}-\|\vd^{k}-\vd^{k+1}\|_\vM^2  \nonumber\\
\leq    & \|\vz^{k}-\vz^*\|^2_{\Lambda^{-1}}+\|\vd^{k}-\vd^*\|_\vM^2 \\ 
-&(1-\max_i{\alpha_iL_i\over 2})\left(\|\vz^{k}-\vz^{k+1}\|^2_{\Lambda^{-1}}+\|\vd^{k}-\vd^{k+1}\|_\vM^2\right). \nonumber
\end{align*}
This completes the proof.
\end{IEEEproof}

\label{pf_lemma_nonincreasing}
\begin{lemma}[Monotonicity of successive difference in a special norm]\label{lemma:nonincreasing}
Let $(\vd^k,\vz^k)$ be the sequence generated from NIDS in~\eqref{for:alg2} with $\alpha_i< 2/L_i$ for all $i$ and $\vI\succcurlyeq c\Lambda^{1/2}(\vI-\vW)\Lambda^{1/2}$.  
Then the sequence $\left\{\|\vz^{k+1}-\vz^k\|^2_{\Lambda^{-1}}+\|\vd^{k+1}-\vd^{k}\|_{\vM}^{2}\right\} _{k\geq0}$ 
is monotonically nonincreasing.
\end{lemma}

\begin{IEEEproof}
Similar to the proof for Lemma~\ref{lemma:powerIneq}, we can show that 
\begin{align}
\begin{split}\label{for:nonincreasing1}
&\langle \vd^{k+1}-\vd^{k},\vz^{k+1}-\vz^{k}+\vx^{k}\rangle\\
&\quad= \langle \vd^{k+1}-\vd^{k},\vd^{k+1}-\vd^{k}\rangle_{\vM},
\end{split}\\
\begin{split} \label{for:nonincreasing2} 
&\langle \vd^{k+1}-\vd^{k},\vz^{k}-\vz^{k-1}+\vx^{k-1}\rangle\\
&\quad= \langle \vd^{k+1}-\vd^{k},\vd^{k}-\vd^{k-1}\rangle_{\vM},
\end{split}  \\
&\langle \vx^{k}-\vx^{k-1},\vz^{k}-\vx^{k}-\vz^{k-1}+\vx^{k-1}\rangle_{\Lambda^{-1}}\geq 0. \label{for:nonincreasing3}
\end{align}
Subtracting~\eqref{for:nonincreasing2} from~\eqref{for:nonincreasing1} on both sides, we have 
\begin{align}
     & \langle \vd^{k+1}-\vd^{k},\vx^{k}-\vx^{k-1}\rangle\nonumber\\
=    & \left\| \vd^{k+1}-\vd^{k}\right\|^2_{\vM}-\langle \vd^{k+1}-\vd^{k},\vd^{k}-\vd^{k-1}\rangle_{\vM}\nonumber\\
&+\langle \vd^{k+1}-\vd^{k},2\vz^{k}-\vz^{k-1}-\vz^{k+1}\rangle\nonumber\\
\geq & \left\| \vd^{k+1}-\vd^{k}\right\|^2_{\vM}-{1\over2}\| \vd^{k+1}-\vd^{k}\|_\vM^2-{1\over 2}\|\vd^{k}-\vd^{k-1}\|^2_{\vM}\nonumber\\
     & +\langle \vd^{k+1}-\vd^{k},2\vz^{k}-\vz^{k-1}-\vz^{k+1}\rangle\nonumber\\
= & {1\over2}\| \vd^{k+1}-\vd^{k}\|_\vM^2-{1\over 2}\|\vd^{k}-\vd^{k-1}\|^2_{\vM}\nonumber\\
&+\langle \vd^{k+1}-\vd^{k},2\vz^{k}-\vz^{k-1}-\vz^{k+1}\rangle,  \label{nonincreasing4}
\end{align}
where the inequality comes from the Cauchy-Schwarz inequality.
Then, the previous inequality, together with~\eqref{for:nonincreasing3} and the Cauchy-Schwarz inequality, gives 
\begin{align*}
        & \langle \vx^{k}-\vx^{k-1},\nabla s(\vx^{k})-\nabla s(\vx^{k-1})\rangle \nonumber\\
\leq	  & \left\langle \vx^{k}-\vx^{k-1},{\Lambda^{-1}}(\vz^{k}-\vx^{k}-\vz^{k-1}+\vx^{k-1}) \right.\nonumber\\
&\left.+\nabla s(\vx^{k})-\nabla s(\vx^{k-1})\right\rangle \nonumber\\
=       & \langle \vx^{k}-\vx^{k-1},{\Lambda^{-1}}(\vz^{k}-\vz^{k+1}-\vz^{k-1}+\vz^{k})\nonumber\\&-\vd^{k+1}+\vd^{k}\rangle
\nonumber\\
\leq    & \langle \vx^{k}-\vx^{k-1},\vz^{k}-\vz^{k+1}-\vz^{k-1}+\vz^{k}\rangle_{\Lambda^{-1}}\nonumber\\
&-\langle \vd^{k+1}-\vd^{k},2\vz^{k}-\vz^{k-1}-\vz^{k+1}\rangle
\nonumber\\
 		    & -{1\over2}\left\| \vd^{k+1}-\vd^{k}\right\|^2_{\vM}+{1\over2}\| \vd^{k}-\vd^{k-1}\|^2_{\vM} \nonumber\\
=       &  \left\langle {\Lambda^{-1}}(\vx^{k}-\vx^{k-1})-\vd^{k+1}+\vd^{k},\right.\nonumber\\
&\left.\vz^{k}-\vz^{k+1}-\vz^{k-1}+\vz^{k}\right\rangle
\nonumber\\
 	 & -{1\over2}\left\| \vd^{k+1}-\vd^{k}\right\|^2_{\vM}+{1\over2}\| \vd^{k}-\vd^{k-1}\|^2_{\vM} \nonumber\\
 	 \end{align*}
 and consequently
\begin{align*}
        & \langle \vx^{k}-\vx^{k-1},\nabla s(\vx^{k})-\nabla s(\vx^{k-1})\rangle \nonumber\\
\leq    & \langle {\Lambda^{-1}}(\vz^{k+1}-\vz^{k})+\nabla s\left(\vx^{k}\right)-\nabla s\left(\vx^{k-1}\right),\nonumber\\
&\vz^{k}-\vz^{k+1}-\vz^{k-1}+\vz^{k}\rangle \nonumber\\
        &  -\frac{1}{2}\left\| \vd^{k+1}-\vd^{k}\right\| _{\vM}^{2}+\frac{1}{2}\left\| \vd^{k}-\vd^{k-1}\right\| _{\vM}^{2} \nonumber\\	
\leq    & \langle \vz^{k+1}-\vz^{k},\vz^{k}-\vz^{k+1}-\vz^{k-1}+\vz^{k}\rangle_{\Lambda^{-1}}\nonumber\\
&+\frac{1}{2}\left\| \vz^{k}-\vz^{k+1}-\vz^{k-1}+\vz^{k}\right\| ^{2}_{\Lambda^{-1}} \nonumber\\
        &  +\frac{1}{2}\left\| \nabla s\left(\vx^{k}\right)-\nabla s\left(\vx^{k-1}\right)\right\|^{2}_{\Lambda} \nonumber\\
&-\frac{1}{2}\left\| \vd^{k+1}-\vd^{k}\right\| _{\vM}^{2}+\frac{1}{2}\left\| \vd^{k}-\vd^{k-1}\right\| _{\vM}^{2} \nonumber\\	
=       &  \frac{1}{2}\left\| \vz^{k}-\vz^{k-1}\right\|_{\Lambda^{-1}}^{2}-\frac{1}{2}\left\| \vz^{k+1}-\vz^{k}\right\|^{2}_{\Lambda^{-1}}\nonumber\\
&+\frac{1}{2}\left\| \nabla s\left(\vx^{k}\right)-\nabla s\left(\vx^{k-1}\right)\right\|^{2}_{\Lambda} \nonumber\\
        &  -\frac{1}{2}\left\| \vd^{k+1}-\vd^{k}\right\| _{\vM}^{2}+\frac{1}{2}\left\| \vd^{k}-\vd^{k-1}\right\| _{\vM}^{2}.
\end{align*}
The three inequalities hold because of~\eqref{for:nonincreasing3},~\eqref{nonincreasing4}, and the Cauchy-Schwarz inequality, respectively. 
The first and third equalities come from~\eqref{for:alg2_z}.
Rearranging the previous inequality, we obtain
\begin{align*}
      & \left\| \vz^{k+1}-\vz^{k}\right\|^{2}_{\Lambda^{-1}} +\left\| \vd^{k+1}-\vd^{k}\right\|_{\vM}^{2} \nonumber\\
\leq	& \left\| \vz^{k}-\vz^{k-1}\right\|^{2}_{\Lambda^{-1}}+\left\| \vd^{k}-\vd^{k-1}\right\| _{\vM}^{2}\nonumber\\
&+\frac{1}{2}\left\| \nabla s\left(\vx^{k}\right)-\nabla s\left(\vx^{k-1}\right)\right\|^{2}_{\Lambda}\\
      & -\langle \vx^{k}-\vx^{k-1},\nabla s(\vx^{k})-\nabla s(\vx^{k-1})\rangle \nonumber\\
\leq  & \left\| \vz^{k}-\vz^{k-1}\right\|^{2}_{\Lambda^{-1}}+\left\| \vd^{k}-\vd^{k-1}\right\| _{\vM}^{2}\nonumber\\
&+\frac{1}{2}\left\| \nabla s\left(\vx^{k}\right)-\nabla s\left(\vx^{k-1}\right)\right\|^{2}_{\Lambda-2\vL^{-1}}\nonumber\\
\leq  & \left\| \vz^{k}-\vz^{k-1}\right\|^{2}_{\Lambda^{-1}}+\left\| \vd^{k}-\vd^{k-1}\right\| _{\vM}^{2}.
\end{align*}
where the second and last inequalities come from~\eqref{for:cocoer} and $\Lambda< 2\vL^{-1}$, respectively. It completes the proof.
\end{IEEEproof}


\setcounter{theorem}{0}
\begin{theorem}[Sublinear rate]
Let $(\vd^k,\vz^k)$ be the sequence generated from NIDS in~\eqref{for:alg2} with $\alpha_i< 2/L_i$ for all $i$ and $\vI\succcurlyeq c\Lambda^{1/2}(\vI-\vW)\Lambda^{1/2}$. 
We have
\begin{align}\label{conv:sublinear_sm}
\begin{split}
&\|\vz^{k}-\vz^{k+1}\|^2_{\Lambda^{-1}}+\|\vd^{k}-\vd^{k+1}\|_{\vM}^{2} \\
&\quad  \leq   \frac{\|\vz^{1}-\vz^{*}\|^2_{\Lambda^{-1}} +\|\vd^{1}-\vd^{*}\|_{\vM}^{2} }{k(1-\max_i{\alpha_iL_i\over 2})},\end{split}\\
&\|\vz^{k}-\vz^{k+1}\|^2_{\Lambda^{-1}}+\|\vd^{k}-\vd^{k+1}\|_{\vM}^{2} = & o\left(\frac{1}{k+1}\right).\nonumber
\end{align}
Furthermore, $(\vd^k,\vz^k)$ converges to a fixed point $(\bar\vd,\bar\vz)$ of iteration~\eqref{for:alg2} and $\bar\vd\in\range(\vI-\vW)$, if $\vI\succ c\Lambda^{1/2}(\vI-\vW)\Lambda^{1/2}$.
\end{theorem}

\begin{IEEEproof}
Lemma~\ref{lemma:nonincreasing} shows that  $ \left\{\|\vz^{k+1}-\vz^k\|^2_{\Lambda^{-1}}+\|\vd^{k+1}-\vd^{k}\|_{\vM}^{2}\right\} _{k\geq0}$ is monotonically nonincreasing. 
Summing up~\eqref{for:powerIneq2} from $1$ to $k$, we have  
\begin{align*}
&  \sum_{j=1}^{k}\left(\|\vz^{j}-\vz^{j+1}\|_{\Lambda^{-1}}^{2}+\|\vd^{j}-\vd^{j+1}\|_{\vM}^{2}\right)  \\
\leq &  {1\over (1-\max_{i}\frac{\alpha_{i}L_{i}}{2})}(\|\vz^{1}-\vz^{*}\|_{\Lambda^{-1}}^{2}+\|\vd^{1}-\vd^{*}\|_{\vM}^{2}\\
&-\|\vz^{k+1}-\vz^{*}\|_{\Lambda^{-1}}^{2}-\|\vd^{k+1}-\vd^{*}\|_{\vM}^{2}).
\end{align*}
Therefore, we have 
\begin{align*}
& \|\vz^{k}-\vz^{k+1}\|_{\Lambda^{-1}}^{2}+\|\vd^{k}-\vd^{k+1}\|_{\vM}^{2} \\
\leq   &{1\over k}\sum_{j=1}^{k}\left(\|\vz^{j}-\vz^{j+1}\|_{\Lambda^{-1}}^{2}+\|\vd^{j}-\vd^{j+1}\|_{\vM}^{2}\right)  \\
\leq &  {1\over k(1-\max_{i}\frac{\alpha_{i}L_{i}}{2})}(\|\vz^{1}-\vz^{*}\|_{\Lambda^{-1}}^{2}+\|\vd^{1}-\vd^{*}\|_{\vM}^{2}),
\end{align*}
and~\cite[Lemma 1]{Davis2016} gives us~\eqref{conv:sublinear_sm}.

When $\vI\succ c\Lambda^{1/2}(\vI-\vW)\Lambda^{1/2}$, inequality~\eqref{for:powerIneq2} shows that the sequence $(\vd^k,\vz^k)$ is bounded, and there exists a convergent subsequence $(\vd^{k_i}, \vz^{k_i})$ such that $(\vd^{k_i}, \vz^{k_i})\rightarrow (\bar\vd,\bar\vz)$. 
Then~\eqref{conv:sublinear_sm} gives the convergence of $(\vd^{k_i+1}, \vz^{k_i+1})$. More specifically, $(\vd^{k_i+1}, \vz^{k_i+1})\rightarrow (\bar\vd,\bar\vz)$. 
Therefore $(\bar\vd,\bar\vz)$ is a fixed point of iteration~\eqref{for:alg2}. 
In addition, because $\vd^k\in\range(\vI-\vW)$ for all $k$, we have $\bar\vd\in\range(\vI-\vW)$. 
Finally Lemma~\ref{lemma:conv} implies the convergence of $(\vd^k,\vz^k)$ to $(\bar\vd,\bar\vz)$.
\end{IEEEproof}

\ifCLASSOPTIONcaptionsoff
  \newpage
\fi

\end{document}